\documentclass[10pt,letterpaper]{article}
\usepackage[letterpaper, margin=1.25in]{geometry}

\usepackage[latin1]{inputenc}
\usepackage{graphicx}
\graphicspath{{images/}{../images/}}
\usepackage{float}
\usepackage{xypic}
\usepackage{amsthm, amsmath, amssymb}
\usepackage{ mathrsfs }
\usepackage{marginnote}
\usepackage{multicol}
\usepackage{subfiles}
\usepackage{algpseudocode}

\usepackage{tikz, pgf, calc}
\usetikzlibrary{arrows,matrix,positioning,fit,calc}
\usepackage[linewidth=1pt]{mdframed}

\usepackage{parskip}

\usepackage{tikz}

\usepackage{cancel}

\newtheorem{theorem}{Theorem}[section]
\theoremstyle{definition}

\theoremstyle{definition}

\newtheorem{lem}{Lemma}[section]
\newtheorem*{rem}{Remark}
\newtheorem{alg}[theorem]{Algorithm}

\newcommand{\R}{\mathbb{R}}

\newcommand{\ukone}{u^{k+1}}
\newcommand{\uk}{u^k}

\newcommand{\ekone}{e^{k+1}}
\newcommand{\ek}{e^k}
\newcommand{\pkone}{p^{k+1}}
\newcommand{\tkone}{T^{k+1}}

\newcommand{\etkone}{e_T^{k+1}}
\newcommand{\etk}{e_T^k}
\newcommand{\utkone}{\tilde{u}^{k+1}}
\newcommand{\ttkone}{\tilde{T}^{k+1}}

\newcommand{\ettkone}{\tilde{e}_T^{k+1}}
\newcommand{\eutkone}{\tilde{e}^{k+1}}

\author{Elizabeth Hawkins\thanks{\small
	School of Mathematical and Statistical Sciences, Clemson University, Clemson, SC, 29364, evhawki@clemson.edu.}
}
\title{Anderson-Picard based nonlinear preconditioning of the Newton iteration for non-isothermal flow simulations}

\begin{document}
\maketitle

\begin{abstract}
We propose, analyze, and test a nonlinear preconditioning technique to improve the Newton iteration for non-isothermal flow simulations. We prove that by first applying an Anderson accelerated Picard step, Newton becomes unconditionally stable (under a uniqueness condition on the data) and its quadratic convergence is retained but has less restrictive sufficient conditions on the Rayleigh number and initial condition's accuracy. Since the Anderson-Picard step decouples the equations in the system, this nonlinear preconditioning adds relatively little extra cost to the Newton iteration (which does not decouple the equations). Our numerical tests illustrate this quadratic convergence and stability on multiple benchmark problems. Furthermore, the tests show convergence for significantly higher Rayleigh number than both Picard and Newton, which illustrates the larger convergence basin of Anderson-Picard based nonlinear preconditioned Newton that the theory predicts. 
\end{abstract}

\section{Introduction}
The non-isothermal flow simulations are a multiphysics model of natural convection, and are commonly used to model flows with non-constant density or temperature. We consider this system, often called the Boussinesq system, on a finite connected domain $\Omega\subset \R^d$ $(d=2,3)$ with $\partial\Omega=\Gamma_1\cup\Gamma_2$, and is given by
\begin{equation}\label{Bouss}
\begin{cases}
u\cdot \nabla u +\nabla p -\nu\Delta u &= f +R_i  (0 ~ T)^T \text{ in }\Omega,\\
\nabla\cdot u&=0 \text{ in }\Omega,\\
u\cdot \nabla T-\kappa \Delta T&=g \text{ in }\Omega,\\
\end{cases}
\end{equation}
with boundary conditions
\begin{equation}\label{BoussBC}
\begin{cases}
u= 0 &\text{ on }\partial\Omega,\\
T=0 &\text{ on }\Gamma_1,\\
\nabla T\cdot n=0 &\text{ on }\Gamma_2.
\end{cases}
\end{equation}
Here, $u$ is the fluid velocity, $p$ is the pressure, $T$ is the temperature (or density), $\nu>0$ is the kinematic viscosity of the fluid, $\kappa>0$ is the thermal diffusivity, $f$ is the external forcing term, and $R_i>0$ is the Richardson number which accounts for the gravitational force and thermal expansion of the fluid. This model can be observed in nature such as in atmospheric models and katabatic winds, and is also observed in industry for natural ventilation, dense gas dispersion, insulation with double pane window, and solar collectors \cite{Songul}.\\ %On an enclosed domain, a rolling core is formed \\

%Classical natural convection problem in fluid mechanics occurs in an enclosed domain [16]. For natural convection in enclosures, a boundary layer forms near the walls. Outside this layer, a rolling core is formed inside the enclosure. The boundary layer and the core could not be considered independent since the core is covered by the layer. There is a coupling between the core and the boundary layer. This coupling is the main reason of the diﬃculty in solving these systems analytically. Thus, numerical methods and experimental analysis are used.\\

The boundary conditions we state above are no-slip for velocity and mixed homogenous Dirichlet and homogenous Neumann for temperature (the latter of which corresponds to perfect insulation), but our results are extendable to most common boundary conditions. Important physical constants are the Reynolds number $Re$, Prandtl number $Pr$ which describes the ratio of kinematic viscosity and thermal diffusivity, and the Rayleigh number $Ra$:
\begin{equation*}
Re=\nu^{-1}~~~~Pr=\frac{\nu}{\kappa} ~~~~ Ra=Ri ~ Re^2 ~ Pr.\\
\end{equation*}
The non-isothermal flow system \eqref{Bouss}-\eqref{BoussBC} is known to be well-posed under a small data condition \cite{Layton89}, however uniqueness is lost for high $Ra$.\\
%$(\alpha_1 +  \nu^{-2}R_i^2  C_p^4+\alpha_2)<1$
%where $\alpha_1:=C_p^2 \nu^{-1}M_1$, $\alpha_2:=C_p^2\kappa^{-1}M_2$, $M_1=C_p(\nu^{-1}\| f\| +  \kappa^{-1}\|g\|)$ and $M_2=\kappa^{-1}C_p\|g\|$. These constants result from the proofs for stability and uniqueness of the non-isothermal flow system that are provided in the Appendix.\\ 

The Picard iteration and Newton iteration are common iterative methods for solving the system \eqref{Bouss}-\eqref{BoussBC}. The Picard iteration for this system is given by 
\begin{equation}\label{PicardIter}
\begin{cases}
\uk\cdot \nabla \ukone +\nabla \pkone -\nu\Delta \ukone &= f +R_i (0 ~\tkone)^T,\\
\nabla\cdot\ukone&=0,\\
\uk\cdot \nabla \tkone-\kappa \Delta \tkone&=g,
\end{cases}
\end{equation}
with boundary conditions
\begin{equation}\label{IterBC}
\begin{cases}
\ukone= 0 &\text{ on }\partial\Omega,\\
\tkone=0 &\text{ on }\Gamma_1,\\
\nabla \tkone\cdot n=0 &\text{ on }\Gamma_2,
\end{cases}
\end{equation}
and is known to admit stable solutions for any problem data (see Section 2).  Furthermore, provided the data satisfies a smallness condition, the solution to Picard is unique and the iteration convergences linearly for any initial guess (see Section 2). Note that the system \eqref{PicardIter} decouples the equations because the heat transport equation is independent of velocity. While this all makes Picard very attractive with respect to efficiency at each iteration, the linear convergence is a potential drawback \cite{PNslow, transport,NC1}.\\

The Newton iteration for the system \eqref{Bouss}-\eqref{BoussBC} is given by
\begin{equation}\label{NewtonIter}
\begin{cases}
\ukone\cdot \nabla \uk + \uk\cdot\nabla\ukone+\nabla \pkone -\nu\Delta \ukone &= f +R_i (0 ~\tkone)^T+\uk\cdot\nabla\uk,\\
\nabla\cdot\ukone&=0,\\
\ukone\cdot \nabla T^k + \uk\cdot\nabla\tkone-\kappa \Delta \tkone&=g+\uk\cdot\nabla T^k,
\end{cases}
\end{equation}
with boundary conditions \eqref{IterBC}. This iteration will be shown in Section 2 to be well-posed and converge quadratically, but only for a very good initial guess and sufficiently small data (the sufficient conditions are more restrictive than for Picard). Newton is also more computationally expensive at each iteration compared to Picard since it does not decouple, and so one must solve fully coupled linear systems at each iteration. \\%We note that it is common in fluid solvers to use Picard to find an initial guess for Newton by running Picard for several iterations.\\

The purpose of this paper is to consider the iteration defined by, at each step, first applying Anderson accelerated Picard and then applying Newton (i.e. Anderson-Picard based nonlinearly preconditioned Newton). Our intent of combining the iterations in this way is to unite the speed of Newton with the stability and robustness of Picard. This method is simple to implement and we show it has an improved robustness compared to Newton (both with respect to stability and the size of the convergence basin), but still converges quadratically. 
%A similar improved robustness was shown both analytically and numerically in \cite{PRTX24} for this type of nonlinear preconditioning applied to the Navier-Stokes equations (NSE), and 
This work thus fits a theme of recent and fundamentally important research on nonlinear preconditioning \cite{Keyes1,Keyes2,Keyes3}.\\

The general idea of applying Anderson-Picard as a preconditioner was first done in 2024 when it was applied to the Navier-Stokes Equations (NSE) \cite{PRTX24}. The purpose of this paper is to extend this important yet easy to implement idea to non-isothermal flow. The difference between NSE and non-isothermal flow simulations is foremost that it is coupled with another equation which incorporates temperature (or density) into the model. This is further complicated by the extra nonlinear terms that arise in this coupled equation. This consequently gives rise to extra difficulties in the analysis and computation, especially for iterative methods like Newton where the equations remain fully coupled. In other words, it is not at all a straightforward extension. Moreover, the analysis provided in this paper will reveal that the convergence basin for Anderson-Picard-Newton is dependent upon only the initial conditions for velocity. Note that for non-preconditioned Newton it depends on both the initial conditions of velocity and temperature (or density). This is concurrent with the accuracy of Anderson-Picard-Newton depending more critically on velocity in a similar manner.\\

%There are many other variations of utilizing both Picard and Newton together, which have improved efficiency and/or robustness of solvers for various problems \cite{Z19,LR98,PMSB12,PP94}. However, these do not combine Picard and Newton in this manner.\\

We show herein analytically and numerically that the (Anderson-)Picard based nonlinearly preconditioned Newton, shortened to (AA-)Picard-Newton, converges quadratically, with better convergence properties than Newton as well as improved stability. We restrict our analysis to Picard-Newton, i.e no Anderson, and note that the results are extendable to AA-Picard-Newton with some extra technical difficulties. We also provide numerical results for both AA-Picard-Newton and Picard-Newton that show a significant improvement compared to Newton alone. \\ 

This paper is arranged as follows. We provide analysis of the well-posedness, stability, and convergence of Picard and Newton iterations in section 2. While much of this is known, the proofs and results will allow for a better comparison and analysis of Picard-Newton. In section 3 we give analysis on the stability and convergence of Picard-Newton. Following this in section 4, we provide numerical results for both AA-Picard-Newton and Picard-Newton applied to some benchmark tests.

\section{Preliminaries}

Let $\Omega\subset\R^d$ ($d=2\text{ or }3$) be connected with boundary $\partial\Omega=\Gamma_1\cup\Gamma_2$ where $meas(\Gamma_1\cap\Gamma_2)=0$.  We denote the $L^2$ inner product and the $L^2$ norm on $\Omega$ as $(\cdot,\cdot)$ and $\|\cdot \|$, respectively. We also denote the solution spaces for velocity, pressure, temperature, and divergence free velocity, as 
\begin{align*}
X&=\{v\in H^1(\Omega): v =0 \text{ on }\partial\Omega\},\\
Q&=\{q\in L^2(\Omega): \int_{\Omega} q=0 \},\\
D&=\{w\in H^1(\Omega): w=0\text{ on }\Gamma_1 \text{ and }\nabla w\cdot n=0\text{ on }\Gamma_2\},\\
V&=\{ v\in X: (\nabla\cdot v, q)=0 ~\forall q\in Q\}.
  \end{align*}
  Let $X^\ast$, $V^\ast$, and $D^\ast$ be the dual spaces of $X$, $V$, and $D$, respectively. With slight abuse of notation, we also use $(\cdot,\cdot)$ to denote the dual pairing of $X$ and $X^\ast$, $V$ and $V^\ast$, and $D$ and $D^\ast$. Then for example, the dual norm of $X$ is defined by
  \begin{equation*}
\|F\|_{X^\ast}:=\sup_{z\in X}   \frac{(F,z)}{\| z\|_X}.
  \end{equation*}
  %This inequality holds similarly for $F\in V^\ast$ with $z\in V$ and $F\in D^\ast$ with $z\in D$. 
  We also recall the Poincar\'e inequality for $z\in X \text{ or } D$, given by
  \begin{equation*}
  \|z\| \leq C_p \|\nabla z\|,
  \end{equation*}
  where $C_p>0$ is a constant depending only on the size of the domain $\Omega$.\\
  
 Define the trilinear functionals $b:X\times X\times X\rightarrow \R$ and $\hat{b}: X\times D\times D\rightarrow \R$ by :
 \begin{align*}
 b(u,v,w)&:=(u\cdot \nabla v,w)+\frac{1}{2}((\nabla\cdot u)v,w),\\
 \hat{b}(u,v,w)&:=(u\cdot \nabla s,t)+\frac{1}{2}((\nabla\cdot u)s,t).
 \end{align*}
 Then $b(u,v,w)$ and $\hat{b}(u,s,t)$ are skew-symmetric, i.e.,
 $$b(u,v,v)= 0 \text{ and } \hat{b}(u,s,s)=0.$$
 Furthermore, if $u$ is divergence free, then $b(u,v,w)= (u\cdot\nabla v,w)$, and thus $(u\cdot\nabla v,v)=0$. %We assume divergence free solutions and finite elements for simplicity of analysis. However, the results in this paper are not impacted by this choice and only minor changes to the analytical upper bounds occur if non-divergence free elements are used. 
 A well known bound for $b$ and $\hat{b}$ that is used herein is \cite{Laytonbook}: $\forall u,v,w \in X$, $s,t\in D$, $\exists C_s>0$ depending only on $|\Omega|$ such that
 \begin{align}\label{bbound}
 |b(u,v,w)|&\leq C_s \|\nabla u\| \|\nabla v\| \|\nabla w\|,\\
  |\hat{b}(u,s,t)|&\leq C_s \|\nabla u\| \|\nabla s\| \|\nabla t\|.
 \end{align}
 The use of $b$ and $\hat{b}$ will allow for $X,Q,D$ to be replaced by finite dimensional subspaces and have the exact same analysis of the paper hold.
 
 \subsection{Anderson-Acceleration}
Anderson Acceleration (AA) is an extrapolation technique which utilizes previously computed iterates to construct the next solution. It is known to improve convergence for linearly convergent fixed point iterations like Picard \cite{PRX19,EPRX20,PR21, TK15, K03,EPRX20}. Therefore, it is a natural to precondition Newton with Anderson-Picard. %which is a fixed point iteration of the form $g_{PN}=g_N \circ g_P$ where $g_N$ and $g_P$ are the solution operators, or fixed point operators, for Newton and Picard respectively. Because Newton converges quadratically, it is most advantageous to apply AA to the Picard step thereby improving the steps computed solutions. 

AA is defined as follows \cite{PRX19}. Let $g:Y\rightarrow Y$ be a fixed point operator for a Hilbert space $Y$.
\begin{alg}
\begin{algorithmic}
\State AA for depth $m\geq 0$
\State Step 0: Choose $x_0\in Y$
\State Step 1: Find $w_1\in Y$ such that $w_1=g(x_0)-x_0$. Set $x_1=x_0+w_1$.
\State Step k: For $k=2,3,...$. Set $m_k=\min\{k-1,m\}$.
\State ~~[a.] Find $w_k=g(x_{k-1})-x_{k-1}$.
\State ~~[b.] Solve the minimization problem for $\{\alpha_j^k\}_{k-m_k}^{k-1}$
$$\min \| \left(1-\sum_{j=k-m_k}^{k-1} \alpha_j^k\right)w_k + \sum_{j=k-m_k}^{k-1}\alpha_j^k w_j \|_Y .$$
\State ~~[c.] Set 
$$x_k=(1-\sum_{j=k-m_k}^{k-1} \alpha_j^k)(x_{k-1}+w_k)+ \sum_{j=k-m_k}^{k-1} \alpha_j^k \alpha_j^k (x_{j-1}+w_j) $$,
\end{algorithmic}
\end{alg}
where $w_j=g(x_{j-1})-x_{j-1}$ may be referred to as the update step or also as the nonlinear residual. Setting $m=0$ returns the original fixed point iteration.

  \subsection{Non-isothermal flow background}
The weak form of the non-isothermal flow system \eqref{Bouss}-\eqref{BoussBC} is: Given $f\in X^\ast$ and $g\in D^\ast$, $\text{find }(u,p,T)\in (X,Q,D) \text{ satisfying}$ $\forall (v,q,w)\in (X,Q,D)$,
\begin{equation*}
\begin{cases}
b(u, u,v)  +(\nabla p, v)+\nu(\nabla u,\nabla v) &= (f,v) +R_i ((0 ~ T)^T,v),\\
(\nabla\cdot u, q)&=0,\\
\hat{b}(u, T,w)+\kappa (\nabla T,\nabla w)&=(g,w).
\end{cases}
\end{equation*}
The spaces $(X,Q)$ satisfy an inf-sup condition, and thus the weak formulation is equivalently simplified using the solution space $V$: Find $(u,T)\in (V,D)$ satisfying $\forall (v,w)\in (V,D)$,
\begin{equation}\label{boussWF}
\begin{cases}
b(u, u,v)  +\nu(\nabla u,\nabla v) &= (f,v) +R_i ((0 ~ T)^T,v),\\
\hat{b}(u, T,w)+\kappa (\nabla T,\nabla w)&=(g,w).
\end{cases}
\end{equation}

\begin{lem}\label{BoussStability}
Any solution to the non-isothermal flow system \eqref{boussWF} satisfies the a priori estimate
\begin{align}
\|\nabla T\|&\leq \kappa^{-1}\|g\|_{D^\ast}=: M_2, \label{TBound}\\
\|\nabla u\|&\leq \nu^{-1}\| f\|_{V^\ast} +R_i C_p^2\nu^{-1} M_2=: M_1\label{uBound}.
\end{align}
%and
%$$\|\nabla u\| + \|\nabla T\| \leq Re\| f\| +  Pr Re\|g\| + C_p^2 Ra \|g\| = M_1 +M_2.$$
\end{lem}

\begin{proof}
We let $v=u$ and $w=T$ in \eqref{boussWF} then using skew-symmetry gives us
\begin{equation}\label{B1}
\begin{cases}
\nu\|\nabla u\|^2 =Ri( (0~T)^T,u)+(f,u),\\
\kappa\|\nabla T\|^2 = (g,T).
\end{cases}
\end{equation}
We upper bound the right hand side terms using the dual space norms, Cauchy-Schwarz, and Poincar\'e, which yields
\begin{align*}
Ri( (0~T)^T,u) &\leq C_p^2 Ri \|\nabla T\| \|\nabla u\|,\\
(f,u) &\leq \|f\|_{V^\ast} \|\nabla u\|, \\
(g,T) &\leq \|g\|_{D^\ast}\|\nabla T\|.
\end{align*}
Using these bounds in \eqref{B1} and reducing provides us with
%\begin{equation*}
%\begin{cases}
%\nu\|\nabla u\|^2 \leq \| f\| \| \nabla u\|+R_iC_p^2|  \|\nabla T\| \|\nabla u\| \\
%\kappa\|\nabla T\|^2 \leq \|g\| \|\nabla T\|.
%\end{cases}
%\end{equation*}
\begin{equation*}
\begin{cases}
\|\nabla u\| \leq \nu^{-1}\| f\|_{V^\ast} +R_i C_p^2\nu^{-1}\|  \nabla T\|, \\
\|\nabla T\| \leq \kappa^{-1}\|g\|_{D^\ast}, 
\end{cases}
\end{equation*}
and using the second of these bounds in the first,
\begin{equation*}
\|\nabla u\| \leq \nu^{-1}\| f\|_{V^\ast} +R_i C_p^2\nu^{-1} \kappa^{-1}\|g\|_{D^\ast}.
\end{equation*}
This proves the result.
\end{proof}

The proof gives the bounds \eqref{TBound} and \eqref{uBound} for $T$ and $u$ respectively, and will be used throughout the paper for solutions to the non-isothermal flow system. These results are similar to those found in \cite{Layton89}, and depending on how the problem is formulated (e.g. Richardson number $\neq 1$ or Prandtl number $\neq 1$), the constants can differ but the results are equivalent. It is known that \eqref{boussWF} admits solutions for any $Ra>0$ \cite{Layton89}, and this can be proven in the same way as the steady Navier-Stokes existence proof using the Leray-Schauder theorem \cite{Laytonbook}. Uniqueness of \eqref{boussWF}, however, requires a smallness assumption on the data. Below is a sufficient condition on the problem data that produces uniqueness.

%Furthermore, existence and uniqueness of solutions to the non-isothermal flow system has been proved under certain conditions on the problem data \cite{BoussExist,Naumann2010}. We assume throughout that the data is such that \eqref{boussWF} has a unique solution.
\begin{lem}
Let $\alpha_1=C_s \nu^{-1} M_1$ and $\alpha_2=C_s \kappa^{-1} M_2$. If $C_p^2 \nu^{-1}R_i , \alpha_2 + \alpha_1<1$, then solutions to \eqref{boussWF} are unique.
\end{lem}
\begin{proof}
Supposing two solutions $(u_1,T_1)$ and $(u_2,T_2)$ to \eqref{boussWF} exist, define $e_u= u_1-u_2$ and $e_T=T_1-T_2$. Now subtracting the systems with these two solutions gives $\forall v\in V, w\in D$,
\begin{equation*}
\begin{cases}
b(u_1,e_u,v)+b(e_u,u_2,v)+\nu(\nabla e_u,\nabla v)&= Ri( (0~e_T)^T,v),\\
\hat{b}(u_1, e_T,w)+\hat{b}(e_u,T_2,w)+\kappa(\nabla e_T,\nabla w)&=0.
\end{cases}
\end{equation*}
Taking $v=e_u$ and $w=e_T$ vanishes two nonlinear terms and leaves
\begin{equation*}
\begin{cases}
\nu\|\nabla e_u\|^2&= Ri( (0~e_T)^T,e_u)-b(e_u,u_2,e_u)\\
&\leq C_p^2 R_i \|\nabla e_T\|\|\nabla e_u\| + C_s\|\nabla e_u\|^2\|\nabla u_2\|,\\
&\\
\kappa\|\nabla e_T\|^2&=-\hat{b}(e_u,T_2,e_T)\\
&\leq C_s \|\nabla T_2\| \|\nabla e_u\|\|\nabla e_T\|.
\end{cases}
\end{equation*}
Next, using the bounds \eqref{uBound} and \eqref{TBound} and simplifying gives
\begin{equation}\label{eq:B3}
\begin{cases}
\|\nabla e_u\|&\leq C_p^2 \nu^{-1}R_i \|\nabla e_T\|+ \alpha_1\|\nabla e_u\|,\\
\|\nabla e_T\|&\leq \alpha_2 \|\nabla e_u\|.
\end{cases}
\end{equation}
Adding these and simplifying results in
\begin{equation*}
(1-\alpha_1-\alpha_2)\|\nabla e_u\|+(1-C_p^2 \nu^{-1}R_i )\|\nabla e_T\|\leq 0.\\
\end{equation*}
This provides the uniqueness of the velocity due to the assumption on the data. With this, uniqueness of the temperature follows immediately from the second bound in \eqref{eq:B3}.
\end{proof}

\subsection{Picard iteration}
The weak formulations for Picard for the non-isothermal flow system takes the form: Find $(\ukone,\tkone)\in V\times D$ satisfying $
\forall (v,w)\in V\times D$,
%and Newton, respectively, for the non-isothermal flow system take the form
\begin{equation}\label{PicardWF}
\begin{cases}
b(\uk, \ukone,v)  +\nu(\nabla \ukone,\nabla v) &= (f,v) +R_i ((0 ~\tkone)^T,v),\\
\hat{b}(\uk, \tkone,w)+\kappa (\nabla \tkone,\nabla w)&=(g,w).
\end{cases}
\end{equation}
%and
Note that the Picard iteration decouples the temperature equation and thus solving \eqref{PicardWF} is a two-step process where one first solves a scalar convection-diffusion problem and then an Oseen problem. Effective preconditioners for these linear systems exist in the literature \cite{benzi,elman:silvester:wathen,HR13,BB12}.

\begin{lem}\label{Lemma:PicardBound}
Any solution to the Picard iteration for the non-isothermal flow system satisfies the a priori estimate: for any $k=1,2...$,
\begin{align*}
%\|\nabla \ukone\| +  \|\nabla \tkone\| \leq \nu^{-1}\| f\| +\kappa^{-1} \|g\| +\nu^{-1}\kappa^{-1}R_i C_p^2\|g\|.
 \|\nabla T^k\|\leq M_2,\\
\|\nabla \uk\| \leq M_1,\\
%\|\nabla \ukone\| +  \|\nabla \tkone\| \leq M_1+M_2.
\end{align*}
\end{lem}
\begin{proof}
These results are proved analogously to those of Lemma \ref{BoussStability}.
\end{proof}

%\begin{proof}
%Let $v=\ukone$ and $w=\tkone$ and using skew symmetry
%\begin{equation*}
%\begin{cases}
%\nu\|\nabla \ukone\|^2 &= (f,\ukone) +R_i ((0 ~\tkone)^T,\ukone),\\
%\kappa \|\nabla \tkone\|^2&=(g,\tkone).
%\end{cases}
%\end{equation*}
%then using Cauchy-Schwarz and Poincar\'e
%\begin{equation*}
%\begin{cases}
%\nu\|\nabla \ukone\|^2 &\leq \| f\| \|\nabla\ukone\| +R_i C_p^2\|\nabla\tkone\| \|\nabla\ukone\|,\\
%\kappa \|\nabla \tkone\|^2&\leq \|g\| \|\nabla \tkone\|,
%\end{cases}
%\end{equation*}
%and dividing by common terms
%\begin{equation*}
%\begin{cases}
%\|\nabla \ukone\| &\leq \nu^{-1}\| f\|  +\nu^{-1}R_i C_p^2\|\nabla \tkone\| ,\\
% \|\nabla \tkone\|&\leq \kappa^{-1}\|g\|.
%\end{cases}
%\end{equation*}
%Then we use the bound
%$$ \|\nabla \tkone\|\leq \kappa^{-1}\|g\|=M_2$$
%in the first equation to get
%\begin{equation*}
%\|\nabla \ukone\| \leq \nu^{-1}\| f\|  +\nu^{-1}\kappa^{-1} R_i C_p^2\|g\| =M_1.\\
%\end{equation*}
%Then we add these equations to get the desired results.
%\end{proof}

\begin{lem}
The Picard iteration \eqref{PicardWF} with data satisfying $\min\{1-\nu^{-1}\frac{C_p^2 R_i}{2} ,1-\kappa^{-1}\frac{C_p^2 R_i}{2} \}>0$, admits a unique solution. 
\end{lem}
\begin{rem}
Due to Lemma \ref{Lemma:PicardBound} and since \eqref{PicardWF} is linear, solution uniqueness can be proven immediately. In the finite dimensional case, this will also imply existence of solutions.
\end{rem}

\begin{proof}
Let $Y= V\times D$ and at iteration $k+1$ define $A: Y\times Y\rightarrow \R$ and $F: Y\rightarrow \R$ by
\begin{align*}
A( (\hat{u} ,\hat{T}), (v,w)):&=b(\uk , \hat{u}, v) +\nu(\nabla \hat{u},\nabla v) + \hat{b}(\uk , \hat{T}, w) + \kappa (\nabla \hat{T},\nabla w)-R_i ((0 ~\hat{T})^T,v), \\
F((v,w))&= (f,v) + (g,w),
\end{align*}
so that the Picard iteration is given by $A( (\hat{u},\hat{T}), (v,w)) = F((v,w))$. Consider $A( (\hat{u},\hat{T}), (v,w))$. Using \eqref{bbound}, Cauchy-Schwarz, and Young's inequality we lower bound the equation as
\begin{align*}
A( (\hat{u},\hat{T}),(\hat{u},\hat{T})) &= b(\uk, \hat{u}, \hat{u}) +\nu\|\nabla \hat{u} \|^2+ \hat{b}(\uk, \hat{T}, \hat{T}) + \kappa \|\nabla \hat{T}\|^2-R_i ((0 ~\hat{T})^T,\hat{u})  \\
&\geq \nu\|\nabla \hat{u}\|^2 +  \kappa \|\nabla \hat{T}\|^2 -\frac{C_p^2 R_i}{2} \|\nabla \hat{T}\|^2 - \frac{C_p^2 R_i}{2}\|\nabla \hat{u}\|^2 \\
&\geq \min\{\nu-\frac{C_p^2 R_i}{2} ,\kappa-\frac{C_p^2 R_i}{2} \}\| (\hat{u},\hat{T})\|^2_Y.
\end{align*}
Hence $A$ is coercive. Continuity of $A$ and $F$ follow easily using the bounds and lemmas above. Thus Lax-Milgram applies and gives existence and uniqueness of \eqref{PicardWF}.\\
%\EVHa{Consider
%\begin{align*}
%A( (u,T), (v,w) )&=b(\uk , u, v) +\nu(\nabla u,\nabla v) + b(\uk, T, w) + \kappa (\nabla T,\nabla w)-R_i ((0 ~T)^T,v)\\
%&\leq (C_s\|\uk\| +\nu)\|\nabla u\| \|\nabla v\| +(C_s \|\uk\|  + \kappa) \|\nabla T\| \|\nabla w\|+ C_p^2 R_i \|\nabla T\| \|\nabla v\| 
%\end{align*}
%Then using the a priori bound on solutions to Picard we upper bound $\|\uk\|\leq M_1$ so that
%\begin{align*}
%A( (u,T), (v,w) )&\leq (C_s M_1+\nu)\|\nabla u\| \|\nabla v\| +(C_s M_1 + \kappa) \|\nabla T\| \|\nabla w\|+ C_p^2 R_i \|\nabla T\| \|\nabla v\| \\
%&\leq \max\{ C_s M_1+\nu, C_s M_1+\kappa, C_p^2 R_i \}( \|\nabla u\| +\|\nabla T\|)(\|\nabla v\| + \|\nabla w\|) \\
%&= \max\{ C_s M_1+\nu, C_s M_1+\kappa, C_p^2 R_i \} \|( u, T)\|_{Y}\| (v, w)\|_{Y}
%\end{align*}
%So $A$ is bounded. Thus Lax-Milgram applies and gives existence and uniqeness.}
\end{proof}

We can now define the solution operator for the Picard iteration. Define $g_P:(V,D)\rightarrow (V,D)$ to be the solution operator for \eqref{PicardWF}: $g_P(u^k,T^k)=(u^{k+1},T^{k+1})$. Because the solutions to Picard are unique, this is a well defined operator.  \\

\begin{lem}\label{Lemma:PicardConv}
%Let $\alpha_1:=C_s\nu^{-1}M_1$ and $\alpha_2:=C_s\kappa^{-1}M_2$. 
Consider the Picard iteration \eqref{PicardWF} with data satisfying $C_p^2\nu^{-1}R_i<1$ and $\alpha_1+\alpha_2 <1$. Then the iteration converges linearly with rate $\alpha_1+\alpha_2$. In particular we have
\begin{equation*}
 \|\nabla (T-T^{k+1}) \| \leq  \alpha_2 \|\nabla(u-u^k)\|,
\end{equation*}
and
\begin{equation*}
\|\nabla (u-u^{k+1})\| \leq (\alpha_1+ \alpha_2) \|\nabla(u-u^k)\|.
\end{equation*}
\end{lem}

\begin{rem}
Note that the sufficient conditions for convergence of Picard are the same as the sufficient conditions for the uniqueness of solutions to the non-isothermal flow system.
\end{rem}

\begin{proof}
We subtract \eqref{boussWF} from \eqref{PicardWF} and choose $v=\ekone$ and $q=\etkone$. Using skew-symmetry , vanishes two nonlinear terms and leaves the equality
\begin{equation*}
\begin{cases}
b(\ek, u,\ekone) +\nu\|\nabla \ekone\|^2 &= R_i  ((0 \hspace{3pt} \etkone)^T,\ekone), \\
\hat{b}(\ek, T,\etkone)+\kappa \|\nabla \etkone\|^2&=0 .
\end{cases}
\end{equation*}
Next we use Cauchy-Schwarz, Poincar\'e, and \eqref{bbound} to upper bound these equations as
\begin{equation*}
\begin{cases}
\nu\|\nabla \ekone\|^2 &\leq C_p^2 R_i \|\nabla\etkone\| \|\nabla\ekone\| + C_s\|\nabla\ek\| \|\nabla u\| \|\nabla\ekone\|,  \\
\kappa \|\nabla \etkone\|^2&\leq C_s\|\nabla\ek\| \|\nabla T\| \|\nabla\etkone\|.
\end{cases}
\end{equation*}
Then we reduce and apply Lemma \ref{BoussStability} to get
\begin{equation*}
\begin{cases}
\|\nabla \ekone\|&\leq C_p^2 \nu^{-1} R_i \|\nabla\etkone\| + \alpha_1\|\nabla\ek\| ,  \\
\|\nabla \etkone\|&\leq \alpha_2 \|\nabla\ek\| .
\end{cases}
\end{equation*}
this gives the bound
\begin{equation*}
 \|\nabla \etkone \| \leq  \alpha_2 \|\nabla\ek\|,
\end{equation*}
Adding the equations and reducing gives
\begin{align*}
\|\nabla \ekone\|+ (1-C_p^2 R_i\nu^{-1})\|\nabla \etkone\|^2 &\leq C_s(\nu^{-1} \|\nabla u\| +\kappa^{-1} \|\nabla T\|)  \|\nabla\ek\|\\
&\leq C_s(\nu^{-1} M_1+\kappa^{-1} M_2)\|\nabla\ek\|\\
& \leq (\alpha_1+\alpha_2) \|\nabla\ek\|.
\end{align*}
%Then we divide by the common term and use the bound from \ref{BoussStability}
%\begin{equation*}
%\begin{cases}
%\|\nabla \ekone\| &\leq C_p^2(\nu^{-1}R_i \|\nabla\etkone\| + \nu^{-1}M_1\|\nabla\ek\|  \\
%\|\nabla \etkone\|&\leq C_p^2\kappa^{-1} M_2\|\nabla\ek\| .
%\end{cases}
%\end{equation*}
%we use the second equation to bound the first
%\begin{equation*}
%\begin{cases}
%\|\nabla \ekone\| &\leq (C_p^2\nu^{-1}R_i \alpha_2+ \alpha_1)\|\nabla\ek\|  \\
%\|\nabla \etkone\|&\leq \alpha_2\|\nabla\ek\| .
%\end{cases}
%\end{equation*}
%add them
%\begin{equation*}
%\|\nabla \ekone\| + \|\nabla \etkone\|\leq (C_p^2\nu^{-1}R_i \alpha_2+ \alpha_1 +\alpha_2)\|\nabla\ek\| .
%\end{equation*}
%where $\alpha_1:=C_p^2\nu^{-1} M_1$ and $\alpha_2:=C_p^2\kappa^{-1} M_2)$.
Finally, using the assumptions on the data finishes the proof.
\end{proof}

\subsection{Newton iteration}
The weak formulation for the Newton iteration for the non-isothermal flow system takes the form: Find $(\ukone,\tkone)\in V\times D$ satisfying $\forall (v,w)\in V\times D$,
\begin{equation}\label{NewtonWF}
\begin{cases}
b(\ukone, \uk,v) + b(\uk,\ukone,v) + \nu(\nabla \ukone,\nabla v) &= (f,v) +R_i ((0 ~\tkone)^T,v)\\
&~~~+b(\uk, \uk,v),\\
\hat{b}(\ukone, T^k,w) + \hat{b}(\uk, \tkone,w)+\kappa (\nabla \tkone,\nabla w)&=(g,w)+\hat{b}(\uk, T^k,w).
\end{cases}
\end{equation}  

\begin{lem}\label{Lemma:N1}
Consider the Newton iteration \eqref{NewtonWF} with data and solutions satisfying 
\begin{align*}
&C_s\nu^{-1}\| \nabla (\uk-u)\| +\alpha_1 + C_s\nu^{-1}\|\nabla (T^k-T)\|+ C_s \nu^{-1}M_2+ C_p^2R_i<1,\\
& \frac{C_s\kappa^{-1} }{4}\|\nabla (T^k-T)\|+\frac{\alpha_2}{4}+ \frac{C_p^2\kappa^{-1} }{4}R_i)<1.
\end{align*}
 Then there exists a unique solution to the iteration. 
\end{lem}

\begin{proof}

We proceed just like for Picard and will use Lax-Milgram. Let $Y= V\times D$ and at iteration $k+1$ define $A: Y\times Y\rightarrow \R$ and $F: Y\rightarrow \R$ by
\begin{align*}
A( (\hat{u},\hat{T} ), (v,w)):&=b(\hat{u},  \uk,v) + b(\uk, \hat{u},v) + \nu(\nabla \hat{u},\nabla v)-R_i ((0 ~\hat{T})^T,v) 
+\hat{b}(\hat{u},  T^k,w) + \hat{b}(\uk, \hat{T},w),\\
&~~~+\kappa (\nabla T,\nabla w)\\
F((v,w))&= (f,v) + b(\uk,\uk,v) + (g,w) + \hat{b}(\uk, T^k,w).
\end{align*}
Then the Newton iteration is given by $A( (\hat{u},\hat{T}), (v,w)) = F((v,w))$. Continuity follows by applying \eqref{bbound}, Poincar\'e, Cauchy-Schwarz, and Young's inequality
\begin{align*}
A( (\hat{u},\hat{T}), (v,w) )&=b(\hat{u},  \uk,v) + b(\uk, \hat{u},v) + \nu(\nabla \hat{u},\nabla v)-R_i ((0 ~\hat{T})^T,v) \\
&
~~~+\hat{b}(\hat{u},  T^k,w) + \hat{b}(\uk, \hat{T},w)+\kappa (\nabla \hat{T},\nabla w)\\
&\leq (2C_s\|\nabla \uk\| + \nu)\|\nabla \hat{u}\| \|\nabla v\| +C_p^2 R_i \|\nabla \hat{T} \| \|\nabla v\| \\
&~~~+C_s\| \nabla \hat{u}\| \|\nabla T^k\| \|\nabla w\| + (C_s \|\nabla \uk\|+\kappa) \|\nabla \hat{T}\| \|\nabla w\|\\
&\leq \sigma_{max}(\|\nabla \hat{u}\| +\|\nabla \hat{T}\|)(\|\nabla v\| +\|\nabla w\|)\\
&\leq \sigma_{max} \| (\hat{u},\hat{T})\|_Y \|(v,w)\|_Y,
\end{align*}
where $\sigma_{max}=\max\{(2C_s \|\nabla \uk\| + \nu), C_p^2R_i, C_s \|\nabla T^k\|, C_s \|\nabla \uk\|  +\kappa \}$. The continuity of $F$ is shown using the dual norms and \eqref{bbound} as
\begin{align*}
F((v,w))&= (f,v) + b(\uk,\uk,v) + (g,w) + \hat{b}(\uk, T^k,w)\\
&\leq \|f\|_{V^\ast}\|v\| + C_s\|\nabla\uk\|^2\|\nabla v\| +\|g\|_{D^\ast}\|\nabla w\| + C_s\|\nabla\uk\|\|\nabla T^k\|\|\nabla w\|\\
&\leq (\|f\|_{V^\ast}+ C_s\|\nabla\uk\|^2+\|g\|_{D^\ast}+ C_s\|\nabla\uk\|\|\nabla T^k\|)\|(v, w)\|_{Y}.
\end{align*}

Next we show $A$ is coercive, and this will require a restriction on the problem data. Here we use Cauchy-Schwarz and skew symmetry to lower bound $A$ as
\begin{align*}
A( (\hat{u},T), (\hat{u},\hat{T}))&=b(\hat{u},  \uk,\hat{u}) + b(\uk, \hat{u},\hat{u}) + \nu\|\nabla \hat{u}\|^2
+\hat{b}(\hat{u},  T^k,T) +b (\uk, \hat{T}, \hat{T})+\kappa \|\nabla \hat{T}\|^2 \\
&~~~- R_i ((0 ~\hat{T})^T,\hat{u}) \\
&\geq b(\hat{u},  \uk,\hat{u}) + \nu\|\nabla \hat{u}\|^2 +\hat{b}(\hat{u},  T^k,\hat{T}) +\kappa \|\nabla \hat{T}\|^2- C_p^2R_i\|\nabla \hat{T}\|\|\nabla \hat{u}\|\\
&\geq b(\hat{u,} \uk-u,\hat{u}) + b(\hat{u},  u,\hat{u})+ \nu\|\nabla \hat{u}\|^2 +\hat{b}(\hat{u}, T^k- T,\hat{T}) \\
&~~~+ \hat{b}(\hat{u}, T,\hat{T})  +\kappa \|\nabla \hat{T}\|^2- C_p^2R_i\|\nabla \hat{T}\|\|\nabla \hat{u}\|.
%&=\int_{\Omega} u\cdot \nabla \uk u + \nu\|\nabla u,\|^2 +\int_{\Omega}u\cdot \nabla T^k T +\kappa \|\nabla T\|^2- C_p^2 R_i\|\nabla T\|\|\nabla u\|\\
\end{align*}

Using the bound \eqref{bbound}
\begin{align*}
b(\hat{u},  \uk-u,\hat{u})&\geq -C_s\|\nabla \hat{u}\|^2\| \nabla (\uk-u)\|, \\
 \hat{b}(\hat{u},  u,\hat{u}) &\geq -C_s\|\nabla \hat{u}\|^2 \|\nabla u\| \geq -C_s M_1\|\nabla \hat{u}\|^2, \\
b (\hat{u},  T^k- T,\hat{T}) &\geq -C_s\|\nabla u\|\nabla (T^k-T)\| \|\nabla \hat{T}\|, \\
 \hat{b}(\hat{u},  T,\hat{T})&\geq -C_s\|\nabla u\|\|\nabla T\|\|\nabla \hat{T}\| \geq -C_s M_2\|\nabla \hat{u}\| \|\nabla \hat{T}\|,
%=\int_{\Omega} u\cdot \nabla (\uk-u) u &= \int_{\partial\Omega} (u (\uk-u) u ) n- 2\int_{\Omega}(\nabla u)\cdot (\uk-u) u \\
%&= -2\int_{\Omega}(\nabla u)\cdot (\uk-u) u\\
%&\geq -2C_p\|\nabla u\|^2\| \uk-u\|  \\
%\int_{\Omega} u \cdot \nabla T^k T &= \int_{\partial\Omega} (u T^k T) n - \int_{\Omega} \nabla u \cdot T^k T - \int_{\Omega} u T^k \cdot \nabla T\\
%&= - \int_{\Omega} \nabla u \cdot T^k T - \int_{\Omega} u T^k \cdot \nabla T\\
%&\geq-C_p\|T^k\|\|\nabla u\|\|\nabla T\|?
\end{align*}
Then we bound $A$ using the same bounds as above to get
\begin{align*}
A( (\hat{u},T), (\hat{u},T))%&=\int_{\Omega} u\cdot \nabla \uk u + \nu\|\nabla u,\|^2 +\int_{\Omega}u\cdot \nabla T^k T +\kappa \|\nabla T\|^2- C_p^2 R_i\|\nabla T\|\|\nabla u\|\\
&\geq (\nu-C_s\| \nabla (\uk-u)\| -C_s M_1)\|\nabla \hat{u}\|^2 \\
&~~~-(C_s\|\nabla (T^k-T)\|+C_sM_2+ C_p^2 R_i) \|\nabla u\|\|\nabla \hat{T}\| +\kappa \|\nabla \hat{T}\|^2\\
&\geq (\nu-C_s\| \nabla (\uk-u)\| -C_s M_1)\|\nabla \hat{u}\|^2 \\
&~~~-(C_s\|\nabla (T^k-T)\|+C_sM_2+ C_p^2 R_i) \|\nabla \hat{u}\|^2\\
&~~~ -(\frac{C_s}{4}\|\nabla (T^k-T)\|+\frac{C_s}{4}M_2+ \frac{C_p^2}{4}R_i) \|\nabla \hat{T}\|^2 +\kappa \|\nabla \hat{T}\|^2\\
&\geq (\nu-(C_s\| \nabla (\uk-u)\| +C_sM_1+C_s\|\nabla (T^k-T)\|+C_sM_2+ C_p^2R_i) )\|\nabla \hat{u}\|^2 \\
&~~~+(\kappa -(\frac{C_s}{4}\|\nabla (T^k-T)\|+\frac{C_s}{4}M_2+ \frac{C_p^2}{4}R_i))\|\nabla \hat{T}\|^2\\
%&\geq \sigma_{min}(\|\nabla \hat{u}\|^2 +\|\nabla \hat{T}\|^2)\\
&\geq\sigma_{min} \|(\hat{u},\hat{T})\|_Y^2,
%&\geq \nu\|\nabla u\|^2 -C_p\|T^k\|\|\nabla u\|\|\nabla T\|+\kappa \|\nabla T\|^2- C_p^2 R_i\|\nabla T\|\|\nabla u\|\\
%&\geq (\nu-\frac{C_p\|T^k\|+ C_p^2 R_i}{2})\|\nabla u\|^2 +(\kappa-\frac{C_p\|T^k\|+C_p^2 R_i}{2}) \|\nabla T\|^2\\
%&\geq \min\{(\nu-\frac{C_p\|T^k\|+ C_p^2 R_i}{2}),(\kappa-\frac{C_p\|T^k\|+C_p^2 R_i}{2})\}(\|\nabla u\|^2_{Y} + \|\nabla T\|_{Y}^2)\
\end{align*}
where 
\begin{multline*}
\sigma_{min}=\min\{(\nu-(C_s\| \nabla (\uk-u)\| +C_sM_1 + C_s\|\nabla (T^k-T)\|+ C_s M_2+ C_p^2R_i) ),\\
(\kappa -(\frac{C_s}{4}\|\nabla (T^k-T)\|+\frac{C_s}{4}M_2+ \frac{C_p^2}{4}R_i))\}.
\end{multline*}
Hence A is coercive. Thus by Lax-Milgram, the Newton iteration \eqref{NewtonWF} is well-posed.

%Let $(u,p,T)$ and $(z,p,F)$ be two solutions to Newton for the non-isothermal flow system. Let $e:=u-z$ and $e_T:=T-F$. Subtracting Newtons iteration at both of these and setting $v=e$ and $w=e_T$ gives
%\begin{equation*}\begin{cases}(e\cdot \nabla \uk,e) + (\uk\cdot\nabla e,e) + \nu\|\nabla e\|^2 &= R_i ((0 ~e_T)^T,e)\\(e\cdot \nabla T^k,e_T) + (\uk\cdot\nabla e_T,e_T)+\kappa \|\nabla e_T\|^2&=0.\end{cases}\end{equation*}  
%Then using skew symmetry
%\begin{equation*}\begin{cases}\nu\|\nabla e\|^2 &= R_i ((0 ~e_T)^T,e)-(e\cdot \nabla \uk,e) \\\kappa \|\nabla e_T\|^2&=-(e\cdot \nabla T^k,e_T).\end{cases}\end{equation*}  
%We use Cauchy-Schwarz, Poincar\'e, and divide by common terms
%\begin{equation*}\begin{cases} \|\nabla e\| &\leq R_i \nu^{-1}C_p^2\| \nabla e_T\| + \nu^{-1}C_p^2\|\nabla e\| \| \nabla \uk\| \\ \|\nabla e_T\|&\leq C_p^2\kappa^{-1} \|\nabla e\| \| \nabla T^k\|.\end{cases}\end{equation*} 
%Then we add the equations and combine like terms
%\begin{equation*}\begin{cases} (1- \nu^{-1}C_p^2\| \nabla \uk\| -C_p^2\kappa^{-1} \| \nabla T^k\|)\|\nabla e\| +  (1-R_i \nu^{-1}C_p^2)\|\nabla e_T\|&\leq  0.\end{cases}\end{equation*}

\end{proof}

With this well-posedness result, the solution operator of \eqref{NewtonWF} defined by $g_N:(V,D)\rightarrow (V,D)$, $g_N(u^k,T^k)=(u^{k+1},T^{k+1})$ is well-defined provided the data restrictions of Lemma \ref{Lemma:N1} are satisfied. We note this is only a sufficient condition, and the Newton iteration for the non-isothermal flow system \eqref{NewtonWF} is believed to be well-posed on a much less restrictive set of parameters. \\

\begin{lem}
Assume for any $k$ that 
\begin{equation*}
\gamma_k:=\alpha_1+\alpha_2+\nu^{-1}C_s\|\nabla (u-u^k)\|+\kappa^{-1}C_s \|\nabla (T-T^k)\|<1,
\end{equation*}
and
\begin{equation*}
 C_p^2 \nu^{-1}R_i<1.
 \end{equation*}
Then the Newton iteration \eqref{NewtonIter} converges quadratically
for any $u^k$ and $T^k$ satisfying
$$C_s(1-\max\{\gamma_k, C_p^2 R_i \nu^{-1}\})^{-1}(\kappa^{-1}+\nu^{-1})(\|\nabla (u-u^k)\|+\|\nabla (T-T^k)\|)<1.$$
%where $\gamma_k:=\max\{\alpha_1+\alpha_2+\nu^{-1}C_s\|\nabla (u-u^k)\|+\kappa^{-1}C_s \|\nabla (T-T^k), C_p^2\nu^{-1}R_i\}$.
\end{lem}

\begin{proof}
Let $\etkone=T-\tkone$ and $\ek=u-u^k$. We subtract \eqref{boussWF} from \eqref{NewtonWF}, add and subtract $b(u,\ekone,v)$ in the first equation and $\hat{b}(\ukone, T,v)$ into the second, and set $v=\ekone$ and $w=\etkone$ to get
\begin{equation*}
\begin{cases}
b(\ekone , \ek,\ekone) &+ b(\ekone,  u,\ekone)+ b(\ek, \ekone,\ekone) +b(u, \ekone,\ekone) +\nu\| \nabla\ekone\|^2 \\
&=R_i ((0~ \etkone)^T,\ekone)+b(\ek,\ek,\ekone), \\
\hat{b}(\ekone,  \etk,\etkone)&+\hat{b}(\ekone,  T,\etkone)+ \hat{b}(\ek,\etkone,\etkone)+\hat{b}(u,\etkone,\etkone)+\kappa \|\nabla \etkone\|^2 \\
&=\hat{b}(\ek, \etk,\etkone).
\end{cases}
\end{equation*}
Next using skew symmetry, 4 terms in the system above vanish, leaving us with 
\begin{equation*}
\begin{cases}
\nu\| \nabla\ekone\|^2 &=R_i ((0 ~\etkone)^T,\ekone)+b(\ek, \ek,\ekone)-b(\ekone,  \ek,\ekone) - b(\ekone , u,\ekone) , \\
\kappa \|\nabla \etkone\|^2 &=\hat{b}(\ek, \etk,\etkone)-\hat{b}(\ekone,  \etk,\etkone)-\hat{b}(\ekone,  T,\etkone).
\end{cases}
\end{equation*}
We use the bound \eqref{bbound}, Cauchy-Schwarz, Poincar\'e, and the stability bounds from Lemma \ref{BoussStability} to upper bound the right hand side terms as
%We use the stability bounds for Boussinesq solutions \ref{BoussStability} to upper bound $u$ and $T$. 
\begin{equation*}
\begin{cases}
 \nu\| \nabla \ekone\|^2 &\leq C_p^2 R_i \| \nabla\etkone\| \|\nabla\ekone\| +C_s \|\nabla\ek\|^2 \|\nabla\ekone\| + C_s \|\nabla\ekone\|^2 \|\nabla \ek\|  + \nu \alpha_1\| \nabla \ekone\|^2,\\
\kappa \|\nabla \etkone\|^2 &\leq C_s\|\nabla\ek\| \|\nabla \etk\| \|\nabla\etkone\| +C_s \|\nabla \ekone\| \| \nabla \etk\| \|\nabla\etkone\|+\kappa \alpha_2 \|\nabla\ekone\| \|\nabla\etkone\|.
\end{cases}
\end{equation*}

%\begin{equation*}
%\begin{cases}
% \| \nabla \ekone\| &\leq C_p^2\nu^{-1} R_i \| \nabla\etkone\| +C_s\nu^{-1}  \|\nabla\ek\|^2 + C_s\nu^{-1}  \|\nabla\ekone\| \|\nabla \ek\|  +  \alpha_1\| \nabla \ekone\|,\\
% \|\nabla \etkone\| &\leq C_s\kappa^{-1}\|\nabla\ek\| \|\nabla \etk\|  +C_s\kappa^{-1} \|\nabla \ekone\| \| \nabla \etk\| + \alpha_2 \|\nabla\ekone\|.
%\end{cases}
%\end{equation*}

This reduces to
\begin{equation}\label{eq:N1}
\begin{cases}
( 1-\alpha_1-\nu^{-1}C_s\|\nabla \ek\| )\| \nabla \ekone\| &\leq C_p^2\nu^{-1}R_i \| \nabla\etkone\|+C_s\nu^{-1}  \|\nabla\ek\|^2  ,\\
~~~~~~~~~~~~~~~~~~~~~~~~~~~~~~~\|\nabla \etkone\| &\leq \kappa^{-1}C_s\|\nabla\ek\| \|\nabla \etk\|+\kappa^{-1}C_s\|\nabla \ekone\| \|\nabla \etk\|+ \alpha_2 \|\nabla\ekone\|.
\end{cases}
\end{equation}
%and after applying Young's inequality
%\begin{equation*}\label{eq:N1}
%\begin{cases}
%( 1-\alpha_1-\nu^{-1}C_s\|\nabla \ek\| )\| \nabla \ekone\| &\leq C_p^2\nu^{-1}R_i \| \nabla\etkone\|+C_s\nu^{-1}  \|\nabla\ek\|^2  ,\\
%\|\nabla \etkone\| &\leq \kappa^{-1}C_s(\frac{1}{2}\|\nabla\ek\^2 + \frac{1}{2}\|\nabla \etk\|^2)+\kappa^{-1}C_s\|\nabla \ekone\| \|\nabla \etk\|+ \alpha_2 \|\nabla\ekone\|,
%\end{cases}
%\end{equation*}

%\begin{equation}\label{eq:N1}
%\begin{cases}
%( 1-\alpha_1-\nu^{-1}C_s\|\nabla \ek\| )\| \nabla \ekone\| &\leq C_p^2\nu^{-1}R_i \| \nabla\etkone\| ,\\
%(1-\alpha_2-C_s\kappa^{-1}\|\nabla\ek\|) \|\nabla \etkone\| &\leq \kappa^{-1}C_s(\frac{1}{2}\|\nabla\ek\|^2 + \frac{1}{2} \|\nabla \etk\|^2+ \|\nabla \etk\|\|\nabla \ekone\|).
%\end{cases}
%\end{equation}
Next we add the equations and reduce to get
\begin{align*}
( 1-\alpha_1-\alpha_2-\nu^{-1}C_s\|\nabla \ek\|-\kappa^{-1}C_s \|\nabla \etk\| )\| \nabla \ekone\| &+ (1-C_p^2\nu^{-1}R_i )\|\nabla \etkone\| \\&\leq C_s\nu^{-1}  \|\nabla\ek\|^2 + \kappa^{-1}C_s\|\nabla\ek\| \|\nabla \etk\|,\\
&\leq C_s(\nu^{-1}  +\kappa^{-1})(\|\nabla\ek\| + \|\nabla \etk\|)^2.
\end{align*}

%\begin{multline*}
%( 1-[\alpha_1+C_s \nu^{-1}\|\nabla \ek\|+C_s \kappa^{-1}\|\nabla\etk\|])\| \nabla \ekone\| + (1-[C_p^2 \nu^{-1}R_i+\alpha_2+C_s \kappa^{-1}\|\nabla\ek\|]) \|\nabla \etkone\|\\
%\leq  \kappa^{-1}C_s\frac{1}{2}(\|\nabla\ek\|^2 + \| \nabla \etk\|^2)
%\end{multline*}
The left hand side is lower bounded as follows,
\begin{multline*}
( 1-\max\{\gamma_k, C_p^2 R_i \nu^{-1}\})(\| \nabla \ekone\| + \|\nabla \etkone\|)\leq ( 1-\alpha_1-\alpha_2-\nu^{-1}C_s\|\nabla \ek\|-\kappa^{-1}C_s \|\nabla \etk\| )\| \nabla \ekone\| \\+ (1-C_p^2\nu^{-1}R_i )\|\nabla \etkone\|.
\end{multline*}
This reduces to,
\begin{equation*}
\| \nabla \ekone\| + \|\nabla \etkone\| \leq C_s( 1-\max\{\gamma_k, C_p^2 R_i \nu^{-1}\})^{-1}( \kappa^{-1}+\nu^{-1})(\|\nabla\ek\| + \| \nabla \etk\|)^2,
\end{equation*}
which completes the proof.
%Then by assumption on $u^0$ and $T^0$ we have $\kappa^{-1}C_p^2( 1-C_p^2\gamma_k)^{-1}(\|\nabla\ek\| + \| \nabla \etk\|)<1$.

\end{proof}

\section{Nonlinear Preconditioner Analysis}
In this section, we consider (AA-)Picard-Newton for the non-isothermal flow system, then analyze its stability and convergence properties. We show that preconditioning stabilizes Newton and reduces the sufficient conditions on small data assumptions and closeness assumption of the initial guess. We restrict our analysis to the case of no AA, however we note that with extra technicalities, similar proofs could be adapted.\\

Using the solution operators for Picard and Newton defined in the previous section, the Picard based nonlinear preconditioning of Newton, Picard-Newton, is defined as $g_{PN}:=g_N\circ g_{P}$ so that 
$$g_{PN}(u^k,T^k)=g_N( g_{P}(u^k,p^k,T^k) )= (u^{k+1}, T^{k+1}).$$
 Thus, Picard-Newton consists of two steps. The first step is the Picard step given by $g_{P}(u^k,T^k)=(\utkone,\ttkone)$ which is then followed by the Newton step given by $g_N(\utkone,\ttkone)=(\ukone,\tkone)$. \\

\begin{alg}{Picard-Newton}
\begin{algorithmic}
\State Given $(u^0,p^0,T^0)\in (X,Q,D)$.
\For{$k=0,1,2,...$} 
\State Step 1: Solve for $(\utkone,\ttkone)\in V\times D$ satisfying
\begin{equation}\label{Picard}
\begin{cases}
b(\uk, \utkone,v)  +\nu(\nabla \utkone,\nabla v) &= (f,v) +R_i ((0 ~\ttkone)^T,v),\\
\hat{b}(\uk, \ttkone,w)+\kappa (\nabla \ttkone,\nabla w)&=(g,w).
\end{cases}
\end{equation}
\State Step 2: Solve for $(\ukone,\tkone)\in  V\times D$ satisfying
\begin{equation}\label{Newton}
\begin{cases}
b(\ukone, \utkone,v) + b(\utkone, \ukone,v) + \nu(\nabla \ukone,\nabla v) &= (f,v) +R_i ((0 ~\tkone)^T,v)\\
&~~~+b(\utkone, \utkone,v),\\
\hat{b}(\ukone, \ttkone,w) + \hat{b}(\utkone,\tkone,w)+\kappa (\nabla \tkone,\nabla w)&=(g,w)+\hat{b}(\utkone, \ttkone,w).
\end{cases}
\end{equation}  \EndFor
\end{algorithmic}
\end{alg}

Picard is known to be unconditionally stable, while Newton is not (and is consequently divergent for many $Ra$). Therefore it is intuitive that the stability of Picard will have a positive effect on the stability of Picard-Newton. Hence, we begin by analyzing the stability of Picard-Newton. %Note that the constants $M_1$, $M_2$, $\alpha_1$, and $\alpha_2$ which will be defined in the following proof can be equivelantly written as  $M_1=C(Re\| f\| +R_i Pr Re \|g\|)$, $M_2=C(Pr Re\| g\|)$, $\alpha_1=C(Re^2\| f\| +Ra \|g\|)$, and $\alpha_2=C(Pr^2 Re^2 \|g\|)$ where $C$ is some constant depending on $\Omega$. We use this equivalent definition in the theorem statement only to clearly show the connection to the $Ra$. % In fact we show that using the stability of Picard, Picard-Newton is stable for certain $Re$ and $Ri$. 
%------------------------------------------------Stability Analysis
\begin{theorem}
If $\alpha_1+\alpha_2, ~\nu^{-1}R_i C_p^2<1$, Picard-Newton \eqref{Picard}-\eqref{Newton} is stable with
\begin{equation*}
(1-\alpha_1-\alpha_2)\|\nabla \ukone\| +  (1-\nu^{-1}R_i C_p)\|\nabla \tkone\| \leq M_2+2M_1.
\end{equation*}
\end{theorem}

\begin{rem}
The small data conditions for stability of Picard-Newton are the same as the uniqueness condition for the non-isothermal flow system.
\end{rem}

\begin{proof}
We begin by considering the stability bound resulting from the Picard step. This part of the proof follows analogously to Lemma \ref{Lemma:PicardBound} resulting in 
%Let $v=\utkone$ and $w=\ttkone$ in the Picard step \eqref{Picard}. Using skew-symmetry we have
%\begin{equation*}
%\begin{cases}\nu\|\nabla \utkone\|^2 &= (f,\utkone) +R_i ((0 ~\ttkone)^T,\utkone),\\\kappa\|\nabla \ttkone\|^2&=(g,\ttkone).\end{cases}
%\end{equation*}
%the right hand side terms are upper bound the using Cauchy-Schwarz and Poincar\'e
%\begin{equation*}
%\begin{cases}\nu\|\nabla \utkone\|^2 &= \|f\|_{V^\ast}\|\nabla\utkone\| +R_i C_p^2\|\nabla\ttkone\| \|\nabla\utkone\|,\\\kappa\|\nabla \ttkone\|^2&=\|g\|_{D^\ast}\|\nabla\ttkone\|,\end{cases}
%\end{equation*}
%and then divide by common factors
%\begin{equation*}
%\begin{cases} \|\nabla \utkone\| &\leq \nu^{-1}\|f\| +\nu^{-1}R_i C_p^2\|\nabla\ttkone\| ,\\\|\nabla \ttkone\|&\leq \kappa^{-1}\| g\|:= M_2.\end{cases}
%\end{equation*}
%We use the second equation to bound the first equation
\begin{equation*}
\begin{cases}
 \|\nabla \utkone\| &\leq M_1  ,\\
 \|\nabla \ttkone\|&\leq M_2.
 \end{cases}
\end{equation*}
%Note that $M_1$ and $M_2$ are the same constants from stability bounds of the usual Picard iteration and the non-isothermal flow system.\\

Next we consider \eqref{Newton} and choose $v=\ukone$ and $w=\tkone$. This gives
\begin{equation*}
\begin{cases}
\nu\|\nabla \ukone\|^2 &= (f,\ukone) +R_i ((0 ~\tkone)^T,\ukone)+b(\utkone, \utkone,\ukone)- b(\ukone, \utkone,\ukone),\\
\kappa \|\nabla \tkone\|^2&=(g,\tkone)+\hat{b}(\utkone,\ttkone,\tkone)- \hat{b}(\ukone, \ttkone,\tkone).
\end{cases}
\end{equation*}

We next use Cauchy-Schwarz, upper bound the trilinear term, use Poincar\'e, and reduce to get
\begin{equation*}
\begin{cases}
\nu\|\nabla \ukone\| &\leq \|f\|_{V^\ast} +R_i C_p^2\|\nabla\tkone\|+C_s\| \nabla\utkone\|^2+ C_s\| \nabla\ukone\| \|\nabla\utkone\|  ,\\
\kappa \|\nabla \tkone\|&\leq \|g\|_{D^\ast}+C_s\|\nabla\utkone\| \|\nabla\ttkone\| + C_s\|\nabla\ukone\| \|\nabla\ttkone\|.
\end{cases}
\end{equation*}

Using the stability bound for the Picard step and combining like terms
\begin{equation*}
\begin{cases}
(1-\alpha_1)\|\nabla \ukone\| &\leq \nu^{-1}\|f\|_{V^\ast} +\nu^{-1}R_i C_p^2\|\nabla\tkone\|+\alpha_1 M_1,\\
 \|\nabla \tkone\|&\leq \kappa^{-1}\|g\|_{D^\ast}+\alpha_2 M_1 + \alpha_2\|\nabla\ukone\|\label{PNTstab},
\end{cases}
\end{equation*}
where $\alpha_1=\nu^{-1}C_s M_1$ and $\alpha_2=\kappa^{-1}C_s M_2$. %Note that these constants also appear in the convergence theorem and proof.

%We use the bound  given by the second equation \eqref{PNTstab} in the first equation 
We add the equations and combine like terms to obtain
\begin{equation*}
(1-\alpha_1-\alpha_2)\|\nabla \ukone\| +  (1-\nu^{-1}R_i C_p)\|\nabla \tkone\| \leq \nu^{-1}\|f\|_{V^\ast} + \kappa^{-1}\|g\|_{D^\ast}+(\alpha_1+\alpha_2) M_1.
\end{equation*}
Using that $M_2=\kappa^{-1}\|g\|_{D^\ast}$ and $M_1=\nu^{-1}\| f\|_{V^\ast} +R_i C_p^2\nu^{-1} M_2$ we can upper bound the left hand side to get the bound
\begin{align*}
(1-\alpha_1-\alpha_2)\|\nabla \ukone\| +  (1-\nu^{-1}R_i C_p)\|\nabla \tkone\| &\leq M_2+(\alpha_1+\alpha_2 +1) M_1,
&\leq M_2+2 M_1.
\end{align*}
\end{proof}

One way to understand the effect of Picard's stability on Newton is to consider the case when $u^k$ does not satisfy the sufficient conditions for convergence of Newton. In this case, one would expect that the Newton step could produce solutions which are divergent. However this is not the case because of Picard's stability bound.\\% $\|\tilde{u}^{k+1}\|\leq M_1$ and $\|\tilde{T}^{k+1}\|\leq M_2$. This stability bound controls the divergent behavior which would normally be exhibited by Newton.\\

%------------------------------------------------Convergence Analysis
Picard-Newton is a two step method and so for clarity we split the error analysis into two theorems: a theorem for the error arising in the Picard step and a theorem for the iterations error after the Newton step. The analysis for the Picard step will follow very closely to the analysis for the usual Picard iteration and thereby includes the constants $M_1$ and $M_2$ that naturally arise in the analysis.
%-----------------------------------------------PN theorem 1
\begin{theorem}
Assume $C_p^2 Ri \nu^{-1} <1$. Then the Picard step \eqref{Picard} of Picard-Newton satisfies
$$\|\nabla (u-\utkone)\| + (1-R_i C_p^2 \nu^{-1})\|\nabla (T-\ttkone) \| \leq (\alpha_1+ \alpha_2) \|\nabla(u-\uk)\|.$$
In particular we have,
\begin{equation}\label{PicBoundT}
 \|\nabla (T-\ttkone) \| \leq  \alpha_2 \|\nabla(u-\uk)\|,
\end{equation}
and
\begin{equation}\label{PicBoundu}
\|\nabla (u-\utkone)\| \leq (\alpha_1+ \alpha_2) \|\nabla(u-\uk)\|.
\end{equation}
\end{theorem}

\begin{proof}
This proof follows analogously to Lemma \ref{Lemma:PicardConv}.

\end{proof}

%\begin{rem}
%We can alter the last step of our analysis to get the bound
%\begin{equation*}
%\|\nabla \eutkone\| \leq (R_i C_p^2 \nu^{-1} \alpha_2+\alpha_1+ \alpha_2) \|\nabla\ek\|\\
%\end{equation*}
%which seemingly removes the condition that $Ri ~\nu^{-1}<1$. However, the condition $Ri~\nu^{-1}$ will appear in the conditions for the Newton step anyway. Furthermore, the condition $Ri ~\nu^{-1}<1$ is part of the well-posedness condition for Picard. Therefore we choose to use the provided bound.
%\end{rem}

%-------------------------------------------------PN theorem 2
Note that the bounds above on $(u-\tilde{u}^{k+1})$ and $(T-\tilde{T}^{k+1})$ both depend upon $(u-u^{k})$. This means the accuracy of $\tilde{T}^{k+1}$ depends on the accuracy of $u^k$ for the Picard step.\\ %This will allow for an important change in the convergence basin by allowing the accuracy of $T^{k+1}$ to depend on $u$. Now we use these bounds to analyze the error of the Newton step \eqref{Newton} which will  give convergence results for the iteration. 

\begin{theorem}
Assume $R_i C_p^2 \nu^{-1}<1$ and 
\begin{equation*}
\beta_k:=\alpha_1+\alpha_2+C_s(\alpha_1+\alpha_2)[\kappa^{-1}+ \nu^{-1}]\|\nabla(u-u^k)\|<1
\end{equation*}
for any $k$. %and $M:=C_p^2(\alpha_1+\alpha_2)(2\nu^{-1}(\alpha_1+\alpha_2) + \kappa^{-1}\alpha_2<1$. 
Then Picard-Newton converge quadratically 
for any $u^k$ satisfying
$$
(1-\max\{\beta_k, R_i C_p^2 \nu^{-1}\})^{-1}C_s (\alpha_1+\alpha_2)^2(\nu^{-1}+ \kappa^{-1}) \|\nabla(u-u^k)\|<1.
$$
%where $\beta_k=\max\{\alpha_1+\alpha_2+C_s(\alpha_1+\alpha_2)[\kappa^{-1}+ \nu^{-1}]\|\nabla(u-u^k)\|, C_p^2R_i \nu^{-1}\}<1$ by assumption.%, and$M= (\alpha_1+\alpha_2)^2(\nu^{-1}+ \kappa^{-1})$ is a constant depending on $Ra$, $\|f\|_{V^\ast}$, and $\|g\|_{D^\ast}$.
\end{theorem}

\begin{proof}
We subtract \eqref{Newton} from \eqref{boussWF}, then add and subtract terms and we the skew symmetry of the trilinear terms to write 
%$$\begin{cases}
%\eutkone \cdot\nabla\ekone&+ u\cdot\nabla\ekone+\eutkone\cdot\nabla u + \ekone\cdot\nabla\eutkone+ \ekone\cdot\nabla u+ u\cdot \nabla\eutkone + u\cdot\nabla u  +\nabla \epkone -\nu\Delta \ekone \\
%&= R_i (0 ~\etkone)^T +\eutkone\cdot\nabla\eutkone +\eutkone\cdot\nabla u+ u\cdot\nabla \eutkone + u\cdot\nabla u\\
%\eutkone\cdot \nabla \etkone &+ u\cdot\nabla\etkone+ \ekone\cdot\nabla\ettkone  +\ekone\cdot\nabla T+ u\cdot\nabla\ettkone + u\cdot \nabla T -\kappa \Delta \etkone + \eutkone\cdot\nabla T\\
%&=\eutkone\cdot\nabla\ettkone + u\cdot\nabla \ettkone +\eutkone\cdot\nabla T+ u\cdot\nabla T
%\end{cases}$$
%$$\begin{cases}\eutkone \cdot\nabla\ekone&+ u\cdot\nabla\ekone + \ekone\cdot\nabla\eutkone+ \ekone\cdot\nabla u  +\nabla \epkone -\nu\Delta \ekone \\&= R_i (0 ~\etkone)^T +\eutkone\cdot\nabla\eutkone  \\\eutkone\cdot \nabla \etkone &+ u\cdot\nabla\etkone+ \ekone\cdot\nabla\ettkone  +\ekone\cdot\nabla T + -\kappa \Delta \etkone \\&=\eutkone\cdot\nabla\ettkone  \end{cases}$$
%where $\etkone=\tkone-T$, $\ekone=\ukone-u$, and $\epkone=\pkone-p$. We multiply by the first equation by$\ekone$ and the second equation by $\etkone$ then integrate oner $\Omega$. Using that $(u,\nabla p)=0$ and the skew symmetry of the trilinear term, the resulting equations simplify to 
$$\begin{cases}
  b(\ekone, \eutkone,\ekone)+ b(\ekone, u,\ekone)  +\nu\|\nabla \ekone \|^2&= (R_i (0 ~\etkone)^T,\ekone) +b(\eutkone, \eutkone,\ekone),  \\
\hat{b}(\ekone, \ettkone,\etkone)  +\hat{b}(\ekone, T,\etkone) +\kappa \|\nabla \etkone \|^2&=\hat{b}(\eutkone, \ettkone,\etkone),
\end{cases}$$
where $\etkone=\tkone-T$ and $\ekone=\ukone-u$.

We now use the bounds from Lemma \ref{BoussStability} to get the upper bounds
$$\begin{cases}
  \nu\|\nabla \ekone \|^2&\leq C_p^2R_i \| \nabla\etkone\| \|\nabla\ekone\| + C_s\|\nabla\eutkone\|^2 \|\nabla\ekone\| + C_s \|\nabla\ekone\|^2\|\nabla\eutkone\|+ \alpha_1 \nu \|\nabla\ekone\|^2 ,  \\
\kappa \|\nabla \etkone \|^2&\leq C_s\| \nabla\eutkone\| \|\nabla\ettkone\| \|\nabla\etkone\|  +C_s\|\nabla\ekone\|\|\nabla\ettkone\|\|\nabla\etkone\|  + \kappa \alpha_2\|\nabla\ekone\| \|\nabla\etkone\| ,
\end{cases}$$
which reduces to
$$\begin{cases}
(1-\alpha_1) \|\nabla \ekone \|&\leq C_p^2  \nu^{-1} R_i \| \nabla\etkone\| + C_s \nu^{-1}\|\nabla\eutkone\|^2 + C_s \nu^{-1} \|\nabla\ekone\|\|\nabla\eutkone\|,  \\
 \|\nabla \etkone \|&\leq C_s\kappa^{-1}\| \nabla\eutkone\| \|\nabla\ettkone\|  +C_s\kappa^{-1}\|\nabla\ekone\|\|\nabla\ettkone\|+  \alpha_2\|\nabla\ekone\| .
\end{cases}$$

Next we continue bounding using \eqref{PicBoundu} and \eqref{PicBoundT} to obtain
$$\begin{cases}
(1-\alpha_1) \|\nabla \ekone \|&\leq C_p^2  \nu^{-1} R_i \| \nabla\etkone\| + C_s \nu^{-1}(\alpha_1+\alpha_2)^2\|\nabla\ek\|^2 + C_s \nu^{-1} (\alpha_1+\alpha_2)\|\nabla\ekone\|\|\nabla\ek\|,  \\
 \|\nabla \etkone \|&\leq C_s\kappa^{-1}\alpha_2(\alpha_1+\alpha_2) \|\nabla\ek\|^2  +C_s\kappa^{-1}\alpha_2\|\nabla\ekone\|\|\nabla\ek\|+  \alpha_2\|\nabla\ekone\| .
\end{cases}$$
%$$\begin{cases}
%  (1-\alpha_1)\|\nabla \ekone \|&\leq C_p^2R_i \nu^{-1}\| \nabla\etkone\| + 2C_s2\nu^{-1}(\alpha_1+\alpha_2)^2\|\nabla\ek\|^2 ,  \\
%\|\nabla \etkone \|&\leq C_s\kappa^{-1}\alpha_2(\alpha_1+\alpha_2) \|\nabla\ek\|^2  +C_s \kappa^{-1}\alpha_2\|\nabla\ekone\|\|\nabla\ek\| +\alpha_2\|\nabla \ekone\|.
%\end{cases}$$

Adding the equations and reducing yields
\begin{align*}
 %&(1-\alpha_1-\alpha_2- C_s \kappa^{-1}\alpha_2\|\nabla\ek\| )\|\nabla \ekone \| + (1-C_p^2R_i \nu^{-1})\|\nabla \etkone \|\\
 %& \leq  C_s(\alpha_1+\alpha_2)(2\nu^{-1}(\alpha_1+\alpha_2) + \kappa^{-1}\alpha_2)\| \nabla\ek\|^2 \\
(1-\alpha_1-\alpha_2-C_s[\kappa^{-1}\alpha_2&+ \nu^{-1} (\alpha_1+\alpha_2)]\|\nabla\ek\|) \|\nabla \ekone \|+(1-C_p^2  \nu^{-1} R_i) \|\nabla \etkone \|\\
&\leq  C_s (\alpha_1+\alpha_2)(\nu^{-1}(\alpha_1+\alpha_2) +\kappa^{-1}\alpha_2) \|\nabla\ek\|^2 \\
 &\leq C_s (\alpha_1+\alpha_2)^2(\nu^{-1}+ \kappa^{-1})\|\nabla\ek\|^2.
 \end{align*}
%where $M= (\alpha_1+\alpha_2)^2(\nu^{-1}+ \kappa^{-1})$.
We now lower bound the left hand side as
\begin{multline*}
(1-\max\{\beta_k, C_p^2 R_i \nu^{-1}\})(\|\nabla \ekone \| +\|\nabla \etkone \| )\leq (1-\alpha_1-\alpha_2-C_s[\kappa^{-1}\alpha_2+ \nu^{-1} (\alpha_1+\alpha_2)]\|\nabla\ek\|) \|\nabla \ekone \| \\
+ (1-C_p^2R_i \nu^{-1})\|\nabla \etkone \|,
\end{multline*}
and upper bound the right hand side to obtain
\begin{equation*}
  \|\nabla \ekone \| +\|\nabla \etkone \|\\
  \leq (1-\max\{\beta_k, C_p^2 R_i \nu^{-1}\})^{-1} C_s (\alpha_1+\alpha_2)^2(\nu^{-1}+ \kappa^{-1})\|\nabla\ek\|^2. %= \tilde{M}\|\nabla\ek\|
\end{equation*}
%where $\tilde{M}= (1-\beta_k)^{-1}C_s M\|\nabla\ek\|$. %\EVH{Then using the assumption on $e^0$ we know $\tilde{M}<1$ for any $k$.} \EVH{Hence we have quadratic convergence with the convergence basin given by $(1-\beta_0)^{-1}M\|\nabla e^0\|$. }

\end{proof}

%$\alpha_1=C(Re^2\| f\| +Ra \|g\|)$, and $\alpha_2=C(Pr^2 Re^{2} \|g\|)$
%\begin{rem}
%Because $\alpha_1+\alpha_2<1$ by asssumption, this makes the convergence basin for $e^0$ larger by allowing more $e^0$ to satisfy the equation
%\begin{equation*}
%C_p^2(\alpha_1+\alpha_2)(2\nu^{-1}(\alpha_1+\alpha_2) + \kappa^{-1}\alpha_2)\| \nabla e^0\|< 1-\beta_0.
%\end{equation*}
%This can be written in terms of $Ra$ and other problem parameters as
%\begin{equation*}
%C_p^2(Re^2\| f\| +Ra \|g\|+Pr^2 Re^2 \|g\|)(2Re(Re^2\| f\| +Ra \|g\|+Pr^2 Re^2 \|g\|) + Pr^3 Re^3 \|g\|)\| \nabla e^0\| < 1-\beta_0
%\end{equation*}
%where $\alpha_1+\alpha_2=Re^2\| f\| +Ra \|g\|+Pr^2 Re^2 \|g\|<1$.
%\end{rem}

We provide table \ref{comparison} below for a more clear comparison of the sufficient conditions for Newton and Picard-Newton. First notice that the assumptions for Picard-Newton include similar terms, and in fact they both require $C_p^2 R_i \nu^{-1}<1$. However, it is clear that Picard-Newtons assumptions are less restrictive as they do not include any assumptions on $T$, and the velocity term in $\beta_k$ in Picard-Newton is scaled by $(\alpha_1+\alpha_2)$ compared to Newton (note that this same condition ensures $(\alpha_1+\alpha_2)<1$). These weakened restriction for Picard-Newton are due to the Picard step allowing $u^k$ to control the error of $T$.\\

Similarly, the sufficient condition on the initial guess for convergence of Picard-Newton is scaled by $(\alpha_1+\alpha_2)^2$ compared to Newton. Because another sufficient condition for Picard-Newton ensures that $\alpha_1+\alpha_2<1$, this allows more $u^k$ to satisfy the equation
\begin{equation*}
C_s(1-\max\{\beta_k, C_p^2 R_i \nu^{-1}\})^{-1}(\alpha_1+\alpha_2)^2(\nu^{-1}+ \kappa^{-1})\| \nabla (u-u^k)\|< 1.
\end{equation*}
%Recall that for the usual Picard, it is required that $\alpha_1+\alpha_2<1$ for convergence. We see a similar restriction on $\alpha_1$ and $\alpha_2$ for the Picard-Newtons convergence basin. 
However there is another important connection to be made between Picard and the sufficient conditions. If we consider just Picard, it is required for its convergence that $\alpha_1+\alpha_2<1$. Then for Picard-Newton,  $\alpha_1$ and $\alpha_2$ come from the error of the Picard step and appear in the sufficient conditions and assumptions of the method. This means that Picard-Newtons conditions for convergence are potentially being aided by Picard's convergence properties. Furthermore, recall that the assumption for uniqueness of solutions to non-isothermal flow system includes $\alpha_1+\alpha_2<1$. This further shows the assumption is reasonable and therefore Picard-Newton's sufficient conditions for convergence of Picard-Newton are improved when compared to Newton.\\

%\EVH{Picard converges for more $Ra$ than Newton however Picard is not universally convergent and coincides with a loss of uniqueness of solutions to the non-isothermal flow system. If the Picard step is not convergent meaning $\alpha_1+\alpha_2\geq 1$ then, as previously discussed, the stability of Picard controls divergent behavior that could be exhibited by the Newton. This leads to iterative solutions which appear to randomly \EVH{?}. This behavior is part of the stability of Picard-Newton and often leads to the iterations of Picard-Newton 'finding' a $u^k$ that is within the convergence basin for the Newton step. This results in convergence.}

 %Furthermore, it is important to notice that the bound on $T-\tilde{T}^{k+}$ from the Picard step removed the convergence basin for $T$ that is seen in the usual Newton iteration and removed the assumptions on $T^k$ that appear in the usual Newton.\\ %At the same time, the bound on $u-\utkone$  and $T-\ttkone$ allowed the convergence basin for Picard-Newton to be heavily dependent upon $\alpha_1$ and $\alpha_2$ which naturally arise in the convergence and stability of Picard. These constants appear only once in the usual Newton through the a priori bound on Boussinesq solutions. 

\begin{table}[H]
\centering
\begin{tabular}{|c|}
\hline
\textbf{Sufficient conditions}\\
\hline
Newton \\ 
\hline
$\gamma_k=\alpha_1+\alpha_2+\nu^{-1}C_s\|\nabla (u-u^k)\|+\kappa^{-1}C_s \|\nabla (T-T^k)\|<1$ \\
 $ C_p^2\nu^{-1}R_i<1$ \\
 \\
 $(1-\max\{\gamma_k,C_p^2 R_i \nu^{-1}\})^{-1}(\kappa^{-1}+\nu^{-1})(\|\nabla (u-u^k)\|+\|\nabla (T-T^k)\|)<1$\\
\\
 \hline
 Picard-Newton\\
 \hline
 $\beta_k=\alpha_1+\alpha_2+C_s(\alpha_1+\alpha_2)[\kappa^{-1}+ \nu^{-1}]\|\nabla(u-u^k)\|<1$\\
  $C_p^2 R_i \nu^{-1} <1$ \\
  \\
   $(1-\max\{\beta_k, C_p^2 R_i\nu^{-1}\})^{-1}(\alpha_1+\alpha_2)^2(\nu^{-1}+ \kappa^{-1}) \|\nabla(u-u^k)\|<1$\\
  \\
 \hline
\end{tabular}
\caption{Shown above is a comparison of sufficient conditions for Newton and Picard-Newton convergence.}\label{comparison}
\end{table}

\section{Numerical Results}
We give now numerical results for two test problems for (AA-)Picard-Newton applied to the non-isothermal flow system. We consider the non-isothermal flow system  with $g=0$, $f=0$, $\nu=\kappa=10^{-1}$, and $Ri$ is varied which varies $Ra$ for the problems. The domain $\Omega$ and boundary conditions are different for each test and will therefore be provided in the following subsections.\\

Let $\tau_h$ be a barycentric refined triangular mesh for $\Omega$. We use $X_h=\mathbb{P}_2(\tau_h) \cap X$, $Q_h=\mathbb{P}_1^{disc}(\tau_h)\cap Q$, and $D_h=\mathbb{P}_2(\tau_h) \cap D$. We note that $(X_h,Q_h)$ are LBB stable in this setting, and also provide divergence free velocity solutions \cite{arnold:qin:scott:vogelius:2D}, so all the analysis of the previous sections applies. We also choose the initial guess to be $(u_0, p_0,T_0)=(0,0,0)$.
Using the B-norm defined below, the convergence criteria will be given as
$$\|(u^k,T^k)-g_{PN}(u^k,T^k)\|_{B}:=\sqrt{\nu \|\nabla(u^k-u^{k-1})\|^2+\kappa \|\nabla(T^k-T^{k-1})\|^2}<10^{-8},$$
and we use 200 iterations as the maximum number of allowed iterations.\\

\subsection{Differentially Heated Cavity}\label{DiffHeat}
The first numerical test is the classical differentially heated cavity problem \cite{Layton89}. The domain is the unit square $\Omega=(0,1)^2$, with boundary conditions
\begin{equation*}
\begin{cases}
u&= 0 ~\text{ on }\partial\Omega,\\
T(0,y)&=0,\\
T(1,y)&=1, \\
\nabla T\cdot n&=0, \text{ on} ~x=0, ~x=1.
\end{cases}
\end{equation*}
%where $g=f=0$ and $\nu=\kappa=10^{-1}$ is fixed. We vary $Ri$ in our numerical tests thereby varying the $Ra$ for the problem. 
We use a mesh with max element diameter $h=1/64$ and note that all results are comparable for other choices of h that we tested. Solutions to this system with $Ra=10000$ are shown in figure \ref{Sol}. We begin the study of the convergence of (AA-)Picard-Newton for varying $Ra$ by considering Picard-Newton and then AA-Picard-Newton.\\

\subsubsection{Picard-Newton}
Convergence results of Picard-Newton are shown in figure \ref{PN} for $Ra=50000$, $100000$, $150000$, $200000$, and $250000$. We observe that Picard-Newton converges for each $Ra$, and quadratic convergence is observed in the asymptotic range. As Ra increases, the number of iterations required for convergence increases as well. We observe the method remains stable even when it is (seemingly) not making progress in its residual. \\

\begin{figure}[H]
\centering
\includegraphics[width=2.5in, height=2.5in, clip, trim=4cm 2cm 2.5cm 1cm]{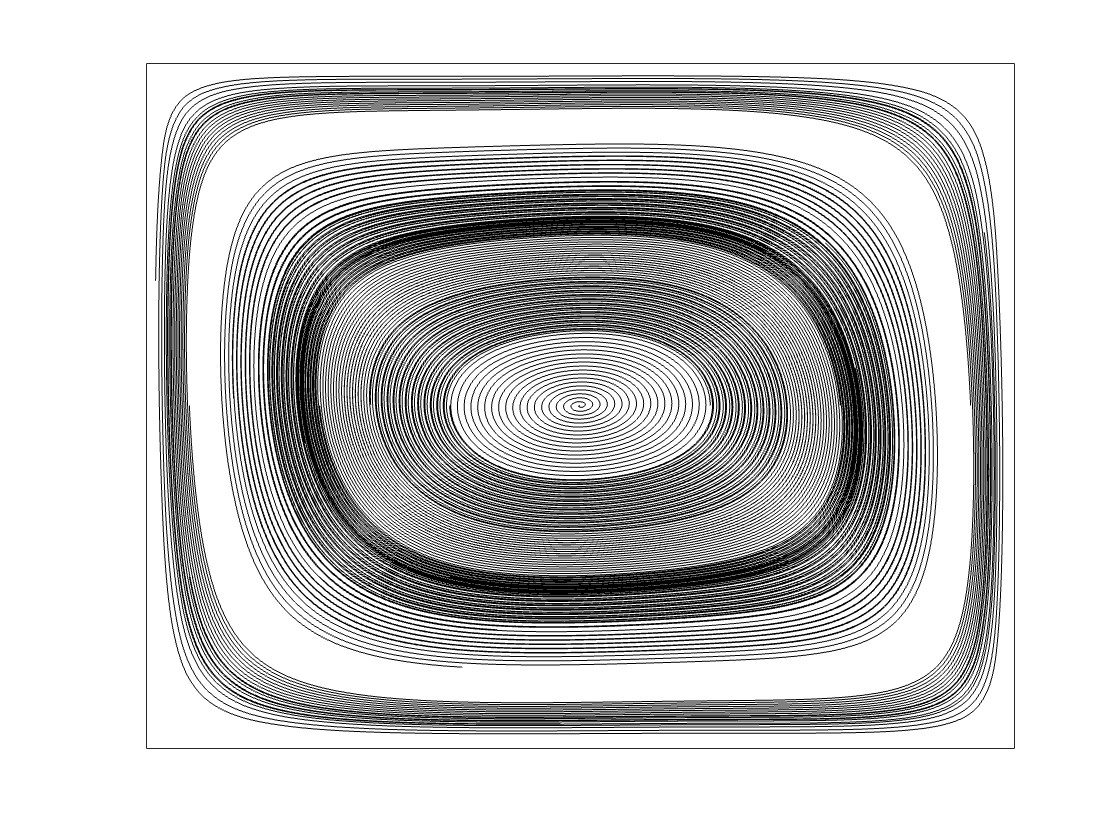}
\includegraphics[width=2.5in, height=2.5in, clip, trim=5.5cm 2cm 4cm 1cm]{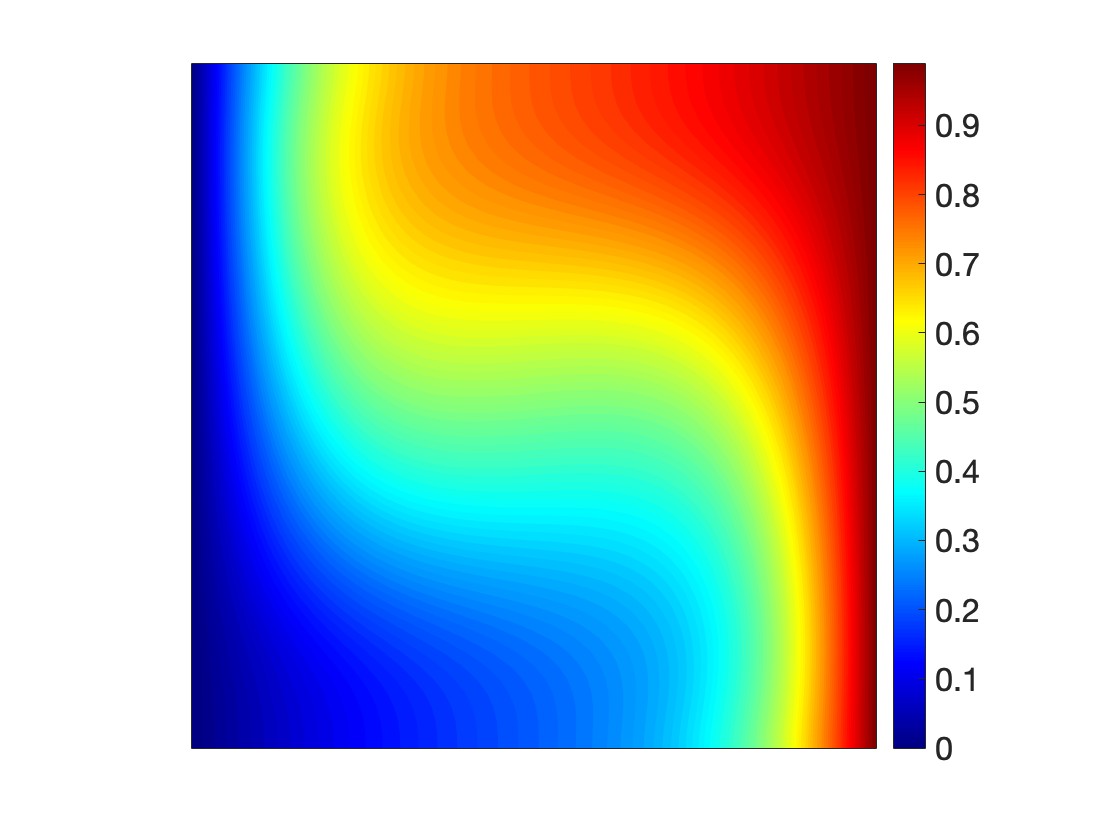}
\caption{Shown above is the computed  solution of the differentially heated cavity problem for velocity streamlines (left) and temperature contours (right) for $Ra=10000$}\label{Sol}
\end{figure}

For comparison, we also perform tests using Picard alone and Newton alone. In figure \ref{P_N_PN} (left) we see that Picard converges within 200 iterations for $Ra$ up to $10000$. For higher $Ra$, Picard does not converge, but it does remain stable. Newton is able to converge for slightly higher $Ra$ than Picard, as shown in figure \ref{P_N_PN} (left): it converges for up to $Ra=15000$, but above this $Ra$ we observe divergence. Moreover, as is typical for Newton, when it diverges, it blows up. %When Newton does converge, if compared to (AA-)Picard-Newton, Newton has slightly higher iteration counts. Furthermore, (AA-)Picard-Newton converges for up to $Ra=250000$ which is an order higher than Newton. 
We observe then that that Picard-Newton is able to converge for $Ra$ a full order of magnitude higher than Picard and Newton.\\

\begin{figure}[H]
\centering
\includegraphics[scale=.19]{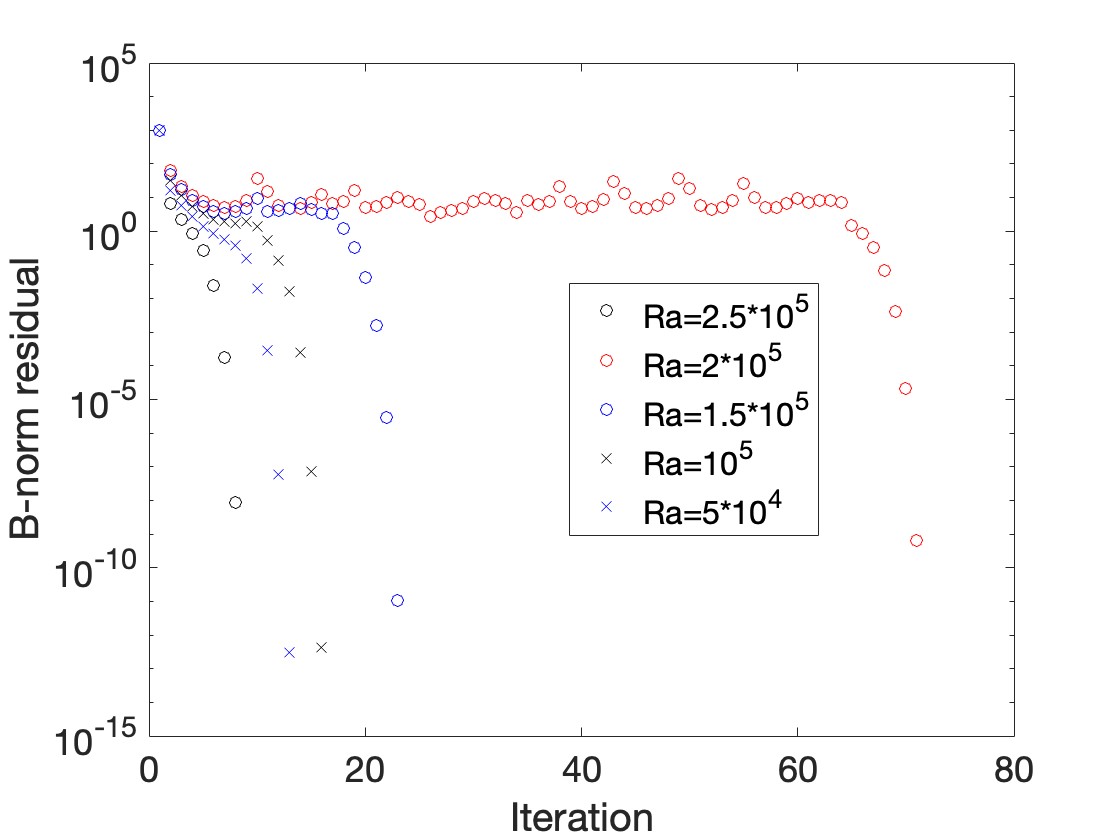}
\caption{Shown above are convergence plots for Picard-Newton at varying $Ra$.}\label{PN}
\end{figure}

\begin{figure}[H]
\centering
\includegraphics[scale=.19]{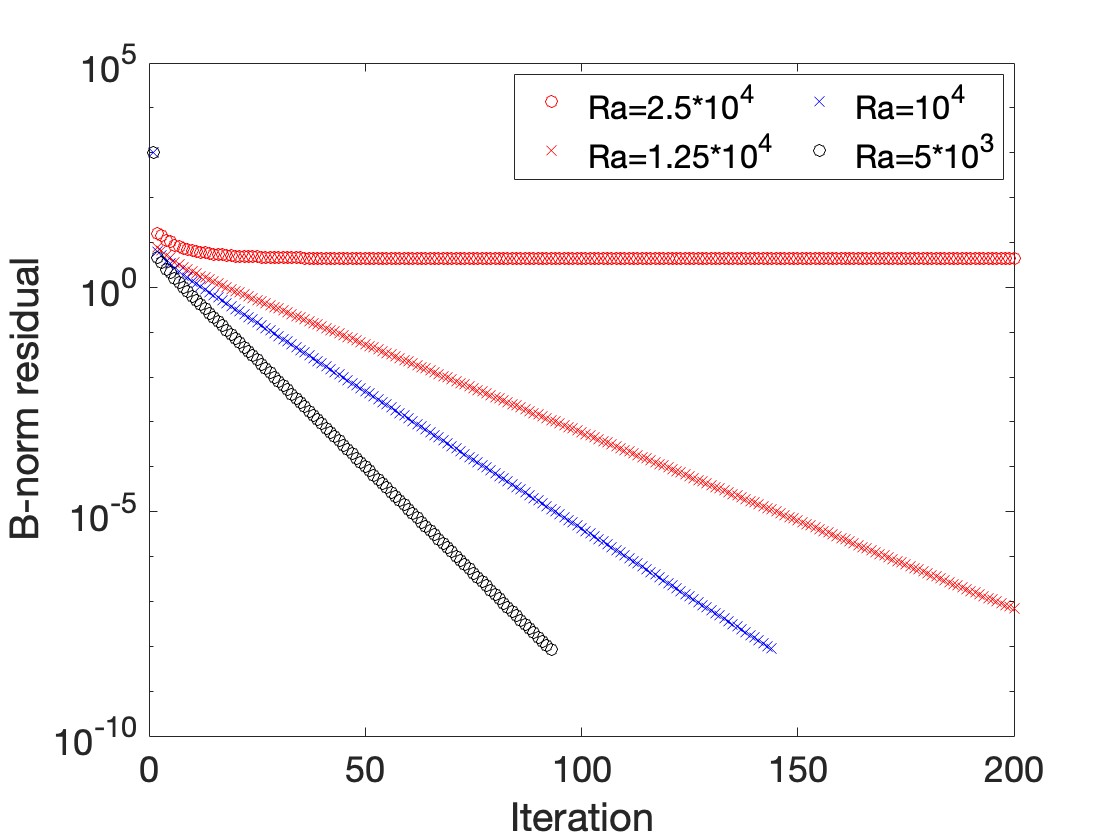}
\includegraphics[scale=.19]{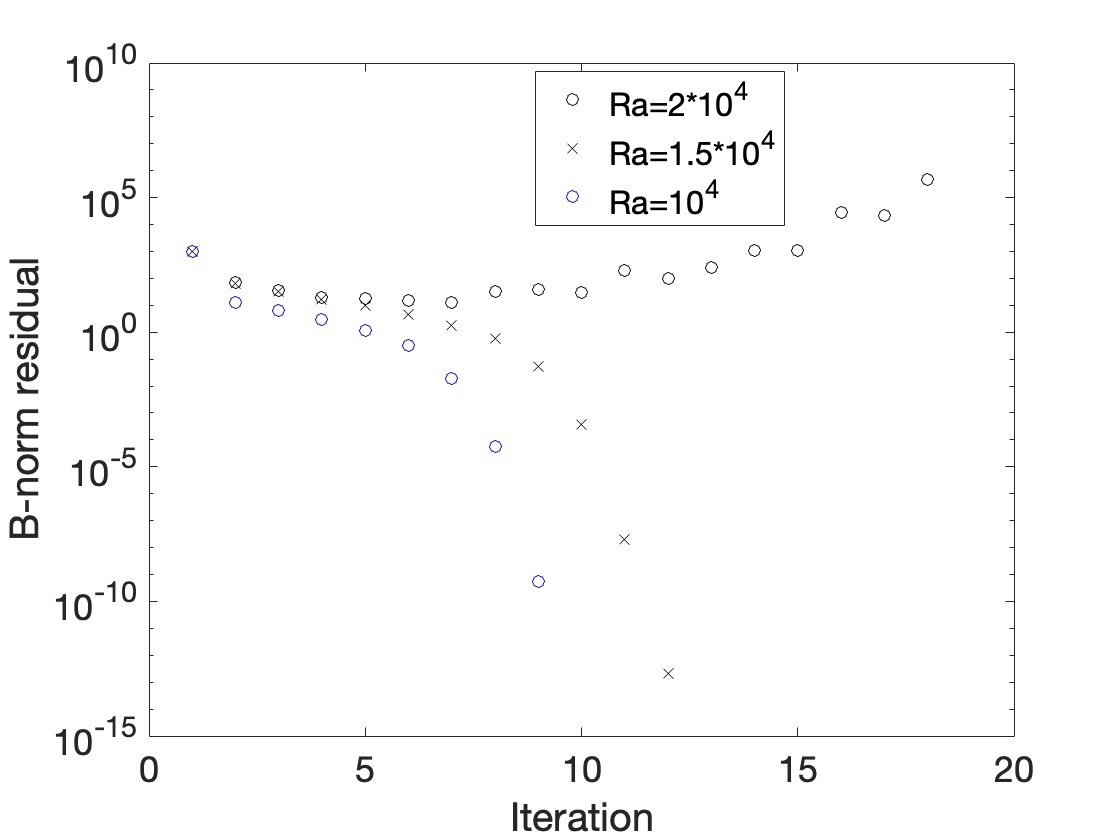}
\caption{Shown above are convergence plots for Picard (left) and Newton(right) at varying $Ra$.}\label{P_N_PN}
\end{figure}

\subsubsection{AA-Picard-Newton for Differentially Heated Cavity}
We now apply AA-Picard-Newton to the differentially heated cavity problem.
%We first apply AA-Picard-Newton to the differentially heated cavity problem in section \ref{DiffHeat}. 
Recall that for this problem Picard-Newton converges for $Ra$ as high as $250000$. Convergence for AA-Picard-Newton are shown in figure \ref{AAfig} for varying $Ra$ and $m=1$ (left) and $m=3$ (right). AA-Picard-Newton shows significant improvement in convergence for higher $Ra$ both in iteration count and solvability. Convergence is even further improved with increased depth: for $m=1$ we observe convergence up to $Ra=500000$ and with $m=3$ convergence is obtained up to $Ra=750000$ (and convergence is accelerated by $m=3$ over $m=1$ for $Ra$ where $m=1$ converges). \\

%Then for $Ra=500000$ we do not see convergence in 200 iterations, but we also do not see divergence. This is a dramatic difference compared to Newton which diverges when it fails. Then for higher depth, given in figure \ref{AAfig}, we get convergence for this $Ra$ and above. \\

\begin{table}[H]\label{table:Ra3}
\centering
\begin{tabular}{|c|c|}
\hline
Method& Max $Ra$ for convergence\\
\hline
AA-Picard-Newton: m=3 & $750000$\\
AA-Picard-Newton: m=1 & $500000$\\
Picard-Newton&$250000$	\\
Picard&$10000$	\\
Newton&$15000$	\\
\hline
\end{tabular}
\caption{Shown above is the maximum $Ra$ (to the nearest 1000) for which each method converges.}
\end{table}

\begin{figure}[H]
\centering
\includegraphics[scale=.19]{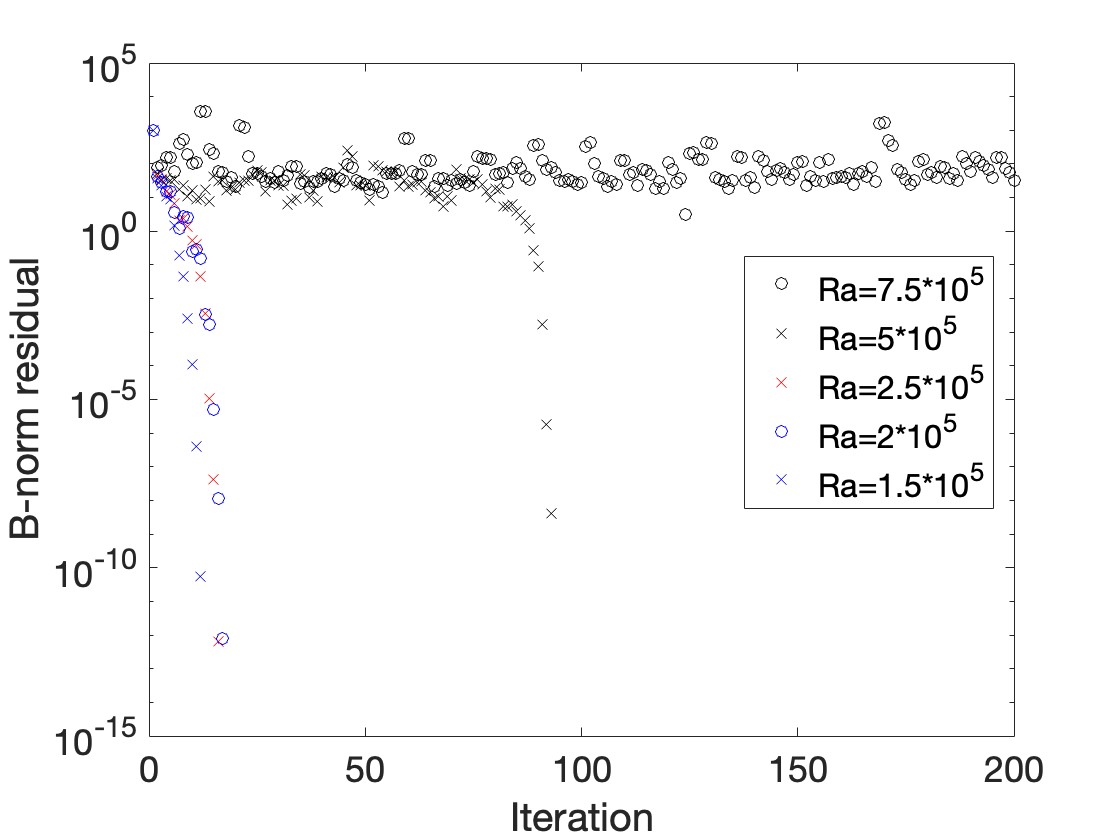}
\includegraphics[scale=.19]{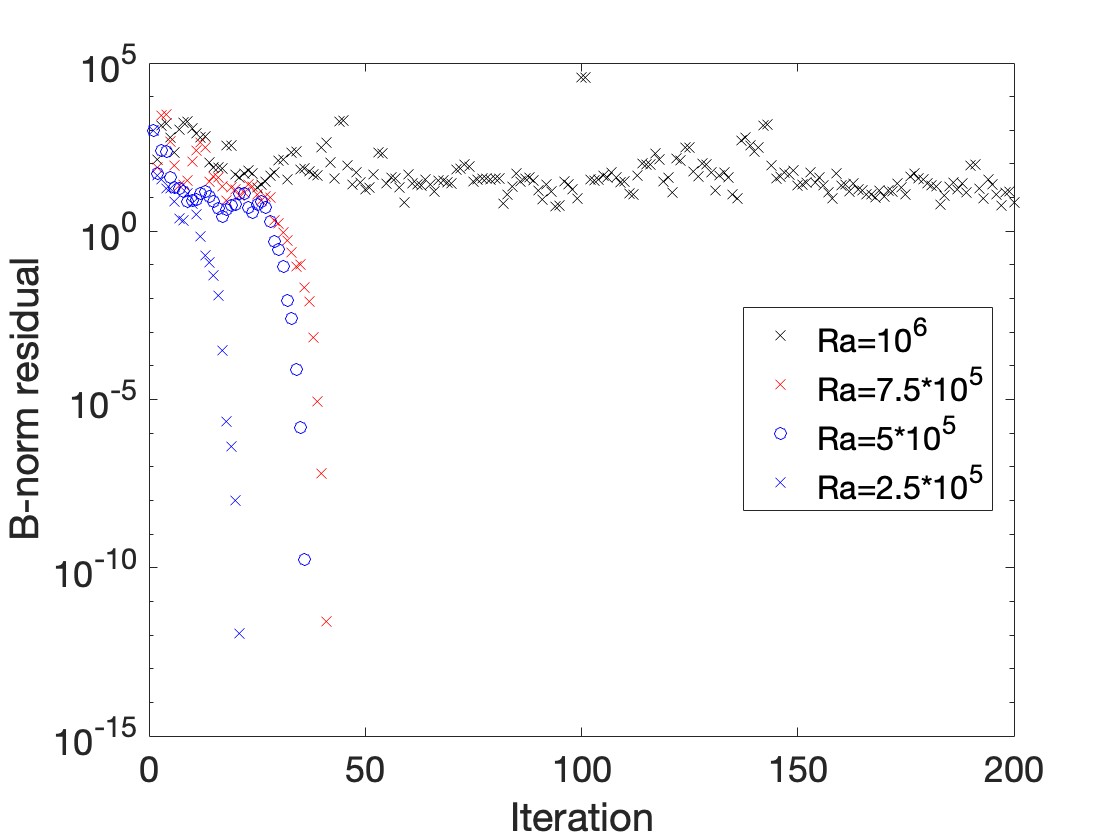}
\caption{Convergence plots for AA-Picard-Newton with $m=1$ (left) and $m=3$ (right) at varying $Ra$.}\label{AAfig}
\end{figure}

\subsection{Heated Cavity on a Complex Domain}\label{Complex}
The second numerical test is on a more complex domain shown in figure \ref{complexSol}, with boundary conditions
\begin{equation*}
\begin{cases}
u&= 0 ~~\text{on }\partial\Omega ,\\
T(x,1)&=1 ,\\
T(x,0)&= \frac{2x}{7},\\
\nabla T\cdot n&=0 ~~0<y<1.
\end{cases}
\end{equation*}
Solutions for $u$ and $T$ shown on the domain in figure \ref{complexSol} for $Ra=1000$. We use a Delaunay mesh with max element diameter $h=7/64$, and again note that all results are comparable for other mesh sizes that we tested. We again begin the study of the convergence of (AA-)Picard-Newton for this heated cavity problem by considering Picard-Newton and then AA-Picard-Newton.

\begin{figure}[H]
\centering
\includegraphics[width=2.5in, height=1.5in ,clip, trim=2cm 4.5cm 1cm 4.5cm]{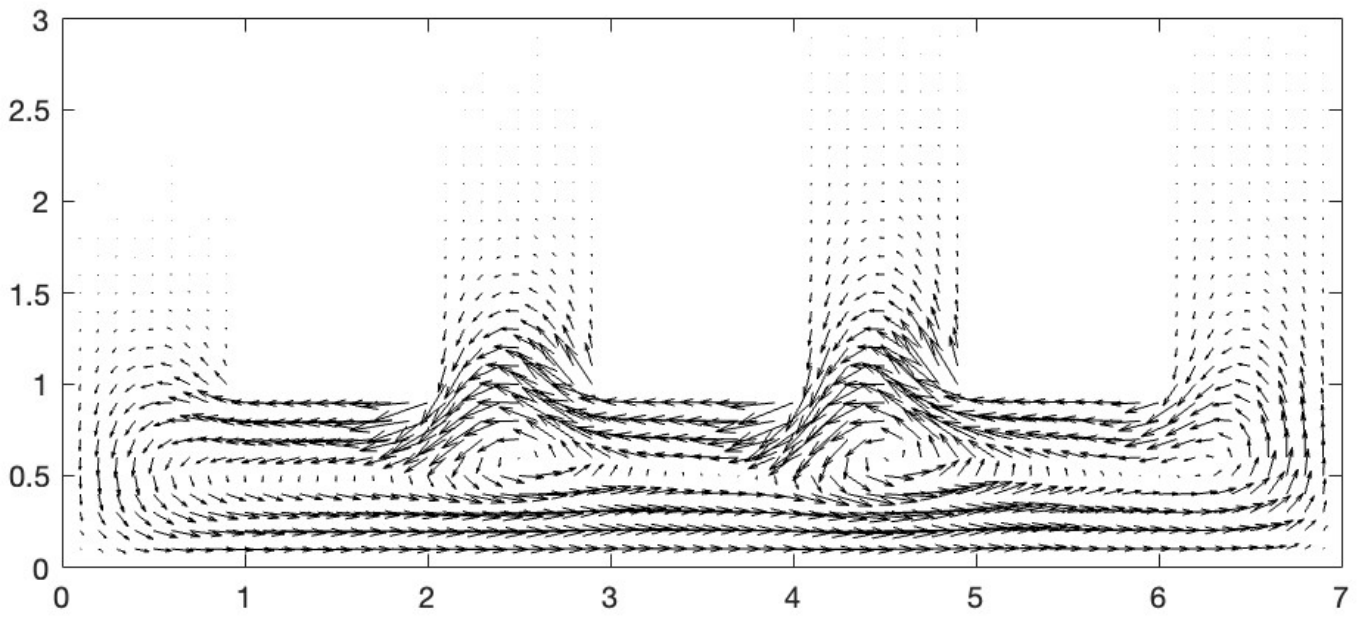}
\includegraphics[width=2.5in, height=1.5in ,clip, trim=2cm 5cm 1cm 5cm]{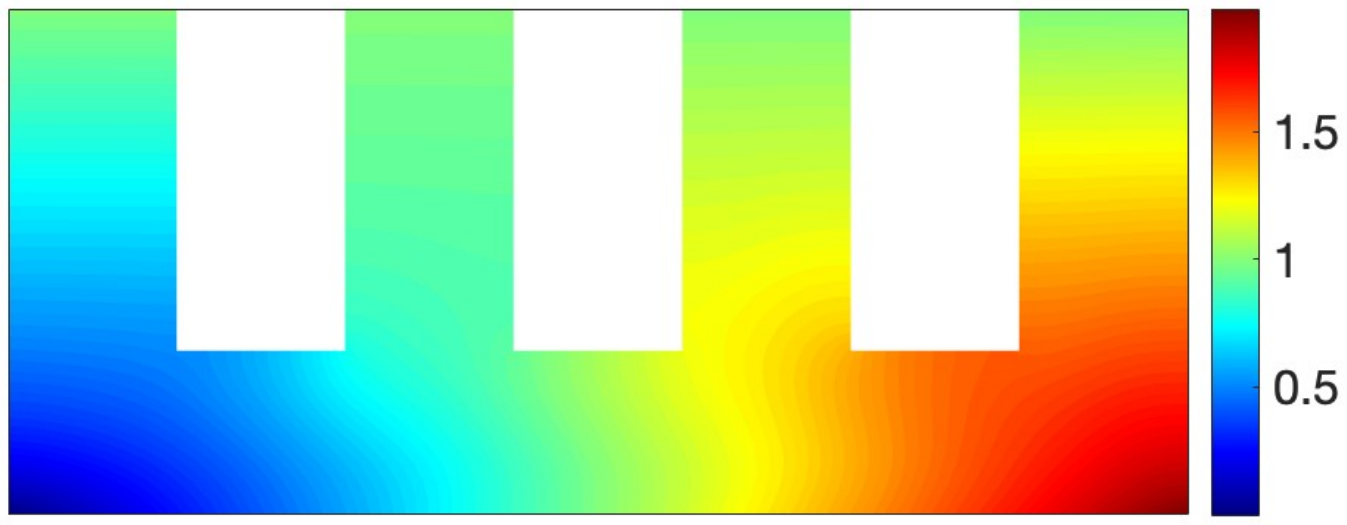}
\caption{Shown above is the computed  solution for $u$ (left) and $T$ (right) for $Ra=1000$ in the second numerical test.}\label{complexSol}
\end{figure}

\subsubsection{Picard-Newton}
The convergence results are shown in figure \ref{P2PN}, for $Ra$ $=$ $5000$, $10000$, $15000$, $20000$, $25000$, \text{ and }$50000$. As in the previous test problem, we observe quadratic convergence for each $Ra$ which Picard-Newton converges, for (up to  $Ra=10000$). For $Ra>10000$, we do not observe convergence within 200 iterations but note the computed solutions remain stable.% (while Newton normally does). This is the stability of (AA-)Picard-Newton.  \\%Similarly to the previous test the convergence is quadratic and often preceded by stability behavior. Furthermore, for $Ra=25000$ (AA-)Picard-Newton simply does not converge within 200 iteration and may converge afterwards. This of course is the unconditional stability of the problem.\\

Again for comparison, we solve the test problem using both Picard alone and Newton alone, and display their convergence results in figure \ref{P2P_N_PN}. We see that Picard converges linearly for $Ra$ up to $1000$. For higher $Ra$, Picard does not converge but remains stable. Newton, in figure \ref{P2P_N_PN} (left), converges for $Ra$ up to $2500$ which is slightly higher than Picard. For $Ra=5000$, convergence is not achieved but Newton curiously remains stable. For $Ra>5000$ we see divergence, and solutions blow up. In comparison, using Picard-Newton we get convergence for $Ra=20000$ with solutions that remain stable for all $Ra$ that we tested. \\
%As seen in figure \ref{P2AA}, (AA-)Picard-Newton with depth $m=1$ and no relaxation, $\beta=1$ converges in 200 iterations for $Ra$ up to $100000$. Then with more depth, $m=3$, (AA-)Picard-Newton $h=1/32$ does converge.

\begin{figure}[H]
\centering
\includegraphics[scale=.19]{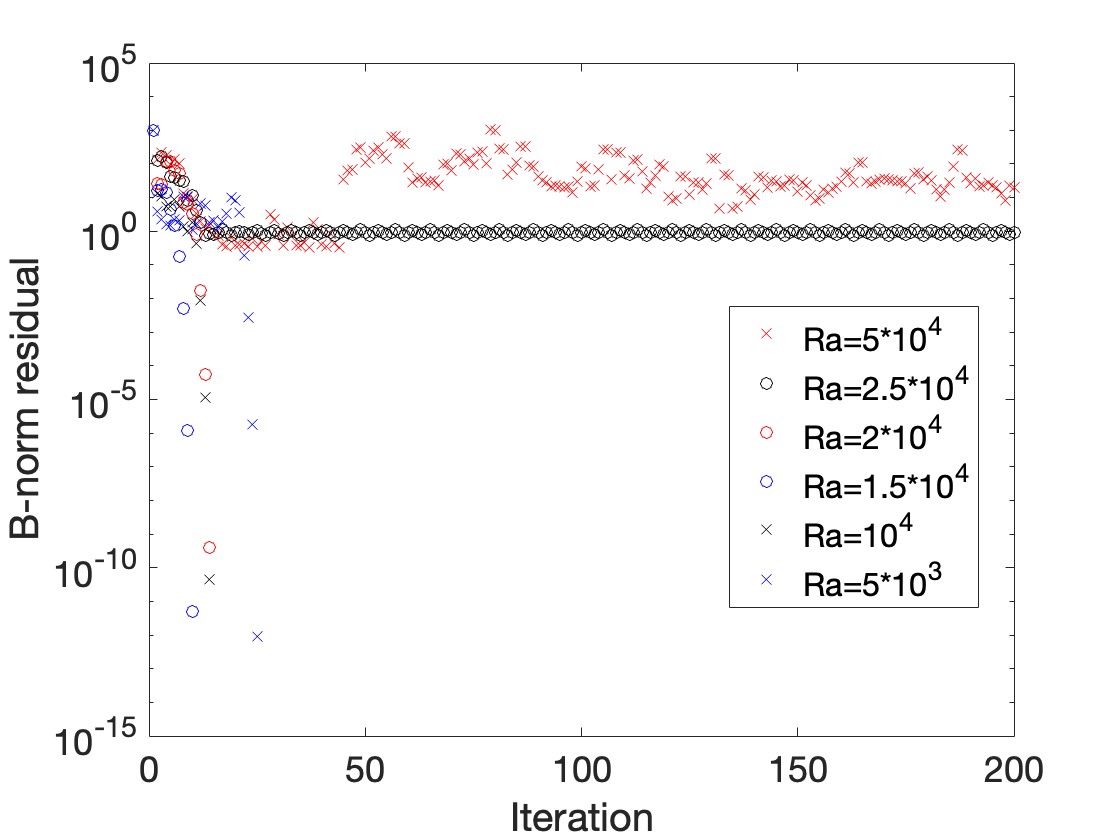}
\caption{Shown above are convergence plots for Picard-Newton at varying $Ra$.}\label{P2PN}
\end{figure}

\begin{figure}[H]
\centering
\includegraphics[scale=.19]{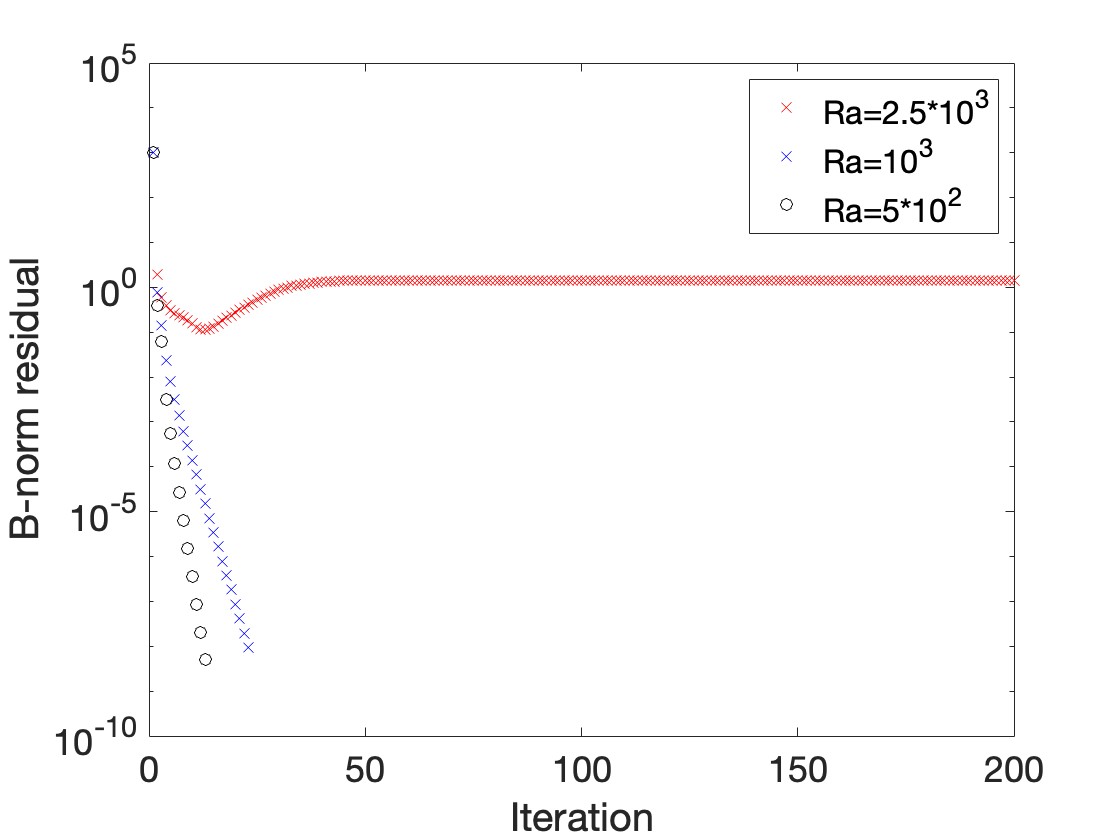}
\includegraphics[scale=.19]{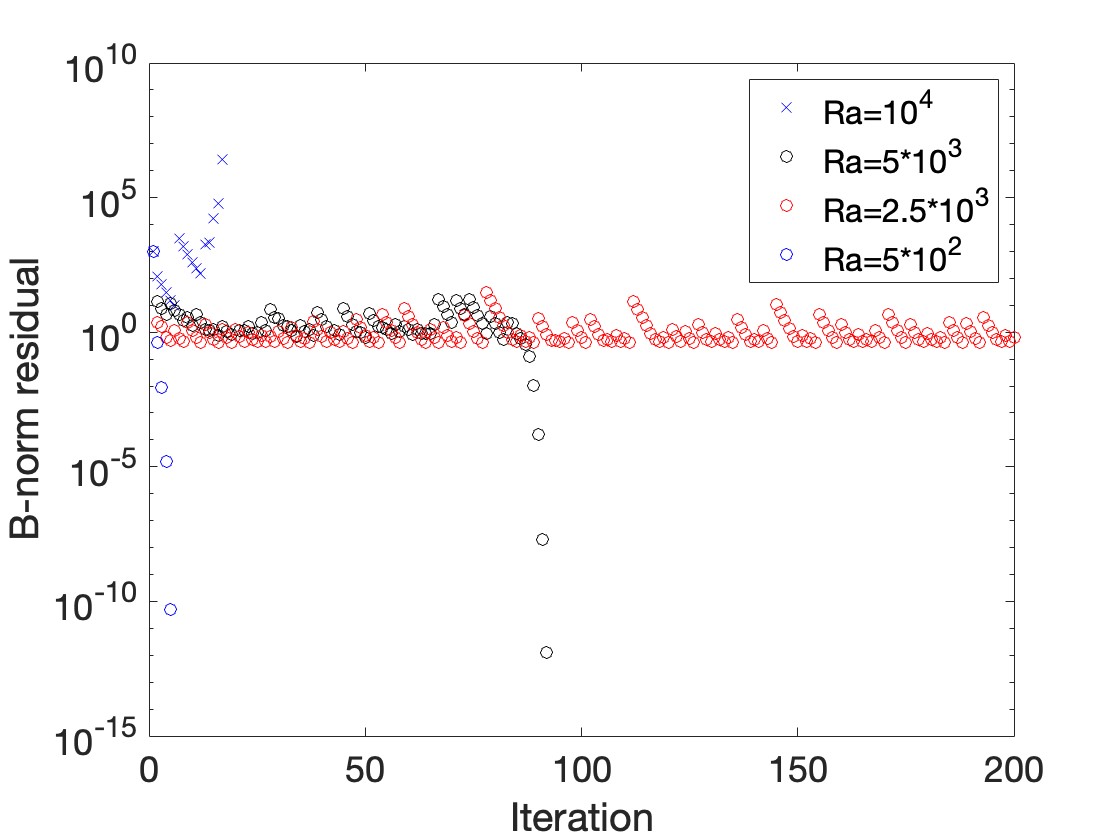}
\caption{Shown above are convergence plots for Picard (left) and Newton (right) at varying $Ra$.}\label{P2P_N_PN}
\end{figure}

\subsubsection{AA-Picard-Newton on a Complex Domain}
We next apply AA-Picard-Newton to the heated cavity problem on a complex domain from section \ref{Complex}. Recall that for this problem (AA-)Picard-Newton converges in 200 iterations for $Ra$ up to $10000$. AA-Picard-Newton achieves convergence with depth $m=1$ for $Ra$ up to $125000$ as shown in figure \ref{P2AA}. This is further improved with depth $m=3$ where AA-Picard-Newton converges  for $Ra$ up to $75000$. As before, for higher $Ra$ we see stability for the method and it may converge if allowed to continue past 200 iterations. \\%Note that $Ra=125000$ is a unique occurrence where the method converges for $m=1$ but not $m=3$ and where a lower $Ra$ does not converge. \\

\begin{table}[H]\label{table:Ra4}
\centering
\begin{tabular}{|c|c|}
\hline
Method& Max $Ra$ for convergence\\
\hline
AA-Picard-Newton: m=3 & $75000$\\
AA-Picard-Newton: m=1 & $125000$\\
Picard-Newton&$20000$	\\
Picard&$1000$	\\
Newton&$5000$	\\
\hline
\end{tabular}
\caption{Shown above is the maximum $Ra$ (to the nearest 100) for which each method converges.}
\end{table}

\begin{figure}[H]
\centering
\includegraphics[scale=.19]{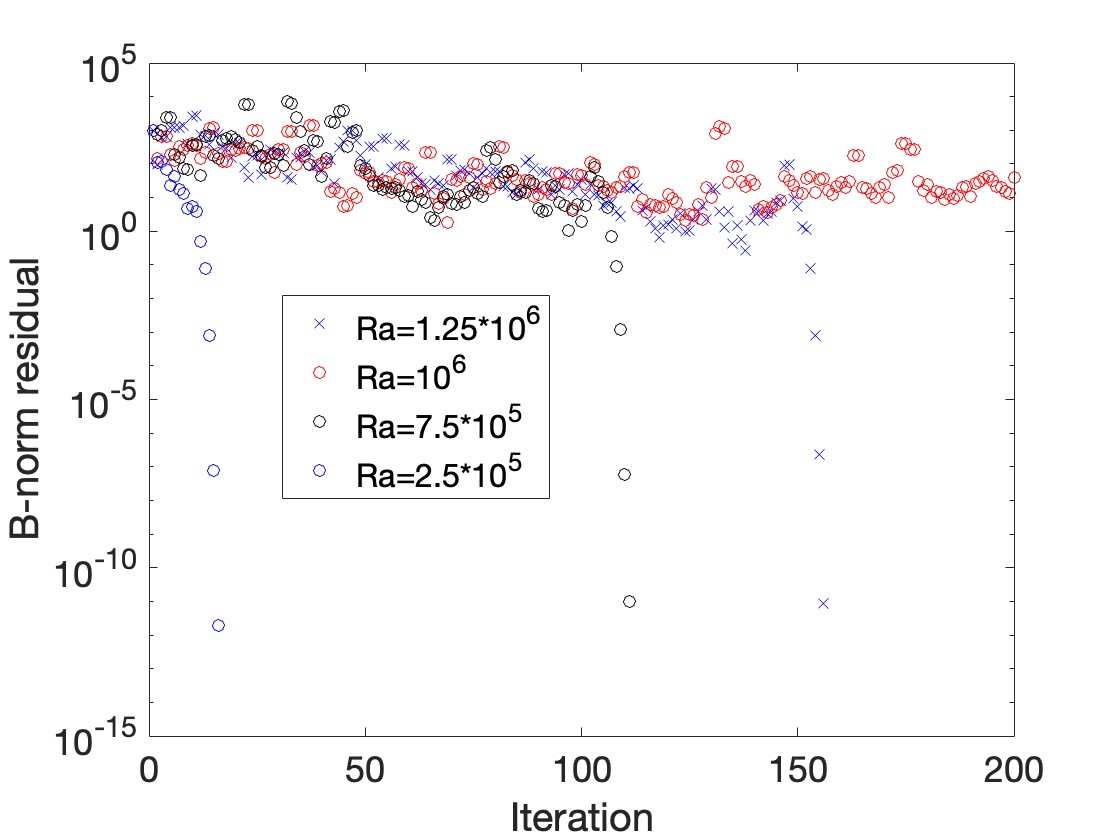}
\includegraphics[scale=.19]{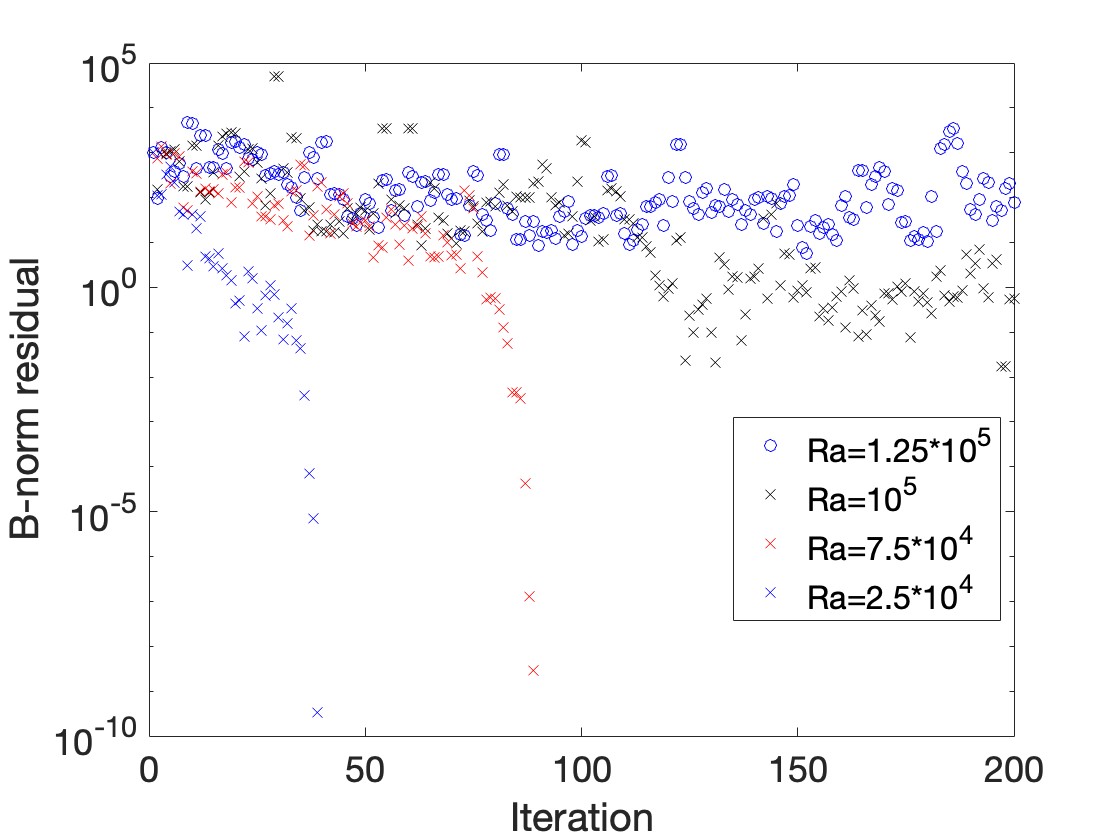}
\caption{Shown above are convergence plots for AA-Picard-Newton convergence with $m=1$ (left) and $m=3$ (right) at varying $Ra$.}\label{P2AA}
\end{figure}

%--------------------------------------------

\section{Conclusion}

(AA-)Picard-Newton for the non-isothermal flow system is an improvement upon the Newton iteration that provides less restrictive (sufficient) conditions for quadratic convergence, and dramatically better convergence results on the benchmark problems. The sufficient conditions of Picard-Newton incorporate the accuracy bounds of the Picard step, which leads to a scaling of the sufficient conditions of Newton by a constant less than $1$. Furthermore, the sufficient conditions of Picard-Newton depend solely upon the fluid velocity $u^k$ due to the preconditioning by Picard causing accuracy of $T^k$ to depend on $u^k$.  This is in stark contrast to the usual Newton whose sufficient conditions depends upon both $u^k$ and $T^k$. Furthermore, Picard-Newton is stable for any initial guess, provided the data satisfy a smallness condition that also implies uniqueness of the PDE system. These analytical results are extendable to AA-Picard-Newton with extra technical difficulties.\\

This improved theory is reflected in the numerical tests where for a simple domain we see convergence for high $Ra$ and errors that remain stable. In comparison, Newton diverges for relatively low $Ra$ and Picard converges slower for the same $Ra$. This is further shown in the non-isothermal flow system on a complex domain which is comparatively harder. (AA-)Picard-Newton converges for high $Ra$ and is proved stable, while Picard and Newton either converges slower or fail to converge for these same $Ra$.

%Lastly, with the popularity of AA, linearly convergent methods like Picard are becoming faster and more applicable to problems where it is usually slow (or unable) to converge. This beneficial effect carries over to (AA-)Picard-Newton: AA-Picard-Newton improved (AA-)Picard-Newton by lowering the iteration count for which the method converges for $Ra$, and also allows convergence for higher $Ra$.

\section{Acknowledgements}
The author was partially supported by NSF grant DMS 2011490.\\ \ \\
The author thanks Professor Leo Rebholz for helpful discussions regarding this work.

\bibliographystyle{plain}
\bibliography{graddiv}

\end{document}